\documentclass[11pt,reqno]{amsart}
\usepackage[margin=1in]{geometry}
\usepackage{amsmath, amsthm, amssymb}
\usepackage[colorlinks=true, pdfstartview=FitV, linkcolor=blue,citecolor=blue, urlcolor=blue]{hyperref}
\usepackage{mathtools}
\usepackage[T1]{fontenc}
\usepackage[utf8]{inputenc}

\usepackage{graphicx}
\usepackage{subcaption}

\usepackage{pdfsync}

\usepackage{parskip} \newtheorem{proposition}{Proposition}[section]
\newtheorem{theorem}[proposition]{Theorem}

\newtheorem{lemma}[proposition]{Lemma}
\theoremstyle{definition}
\newtheorem{definition}[proposition]{Definition}

\numberwithin{equation}{section}

 \def\R{\mathbb{R}}
  \def\Z{\mathbb{Z}}
  
  \def\C{\mathbb{C}}

  \def\ve{\varepsilon}

\newcommand{\ee}{\mathrm{e}}
\newcommand{\p}{\partial}

\newcommand{\HH}{\mathcal{H}}
\newcommand{\KK}{\mathcal{K}}
\newcommand{\JJ}{\mathcal{J}}

\renewcommand{\d}{\mathrm{d}}

\newcommand{\re}{\mathrm{Re}}
\newcommand{\im}{\mathrm{Im}}
\newcommand{\res}{\mathrm{res}}
\newcommand{\eqnb}{\begin{equation}}
\newcommand{\eqne}{\end{equation}}

\renewcommand*{\thefootnote}{(\arabic{footnote})}

\newcommand\blfootnote[1]{%
  \begingroup
  \renewcommand\thefootnote{}\footnote{#1}%
  \addtocounter{footnote}{-1}%
  \endgroup
}

\title[Linear instability of logarithmic spiral vortex sheets]{Linear instability of symmetric logarithmic spiral vortex sheets}

\author{Tomasz Cie\'{s}lak, Piotr Kokocki and Wojciech S. O\.za\'nski}

\begin{document}

\begin{abstract}
We consider Alexander spirals with $M\geq 3$ branches, that is symmetric logarithmic spiral vortex sheets. We show that such vortex sheets are linearly unstable in the $L^\infty$ (Kelvin-Helmholtz) sense, as solutions to the Birkhoff-Rott equation. To this end we consider Fourier modes in a logarithmic variable to identify unstable solutions with polynomial growth in time.
\end{abstract}

\maketitle
\blfootnote{T.~Cie\'{s}lak: Institute of Mathematics, Polish Academy of Sciences, \'Sniadeckich 8, 00-656 Warszawa, Poland, email: cieslak@impan.pl

P.~Kokocki: Faculty of Mathematics and Computer Science, Nicolaus Copernicus University, Chopina 12/18, 87-100 Toru\'n, Poland, email: pkokocki@mat.umk.pl

W. S. O\.za\'nski: Department of Mathematics, Florida State University, Tallahassee, FL 32306 e-mail: wozanski@fsu.edu }

\section{Introduction}

We consider $M\geq 3$ concentric logarithmic spiral vortex sheets, parametrized by
\eqnb\label{log_spiral}
\begin{cases}
Z_{m}(\theta , t) = t^\mu \ee^{a (\theta - \theta_{m})} \ee^{i\theta },\\
\Gamma_{m} (\theta , t) = g_{m} t^{2\mu -1 }  \ee^{2a (\theta - \theta_{m} )} ,\\
r_{m}(\theta , t)  = | Z_{m} (\theta , t )| = t^\mu \ee^{a (\theta - \theta_{m} )},\quad \theta \in \R, \ t>0,
\end{cases}
\eqne
where $a>0$, $\mu\in\R$, $g_{m}\in\R\setminus\{0\}$ and $0=\theta_{0}<\theta_{1}<\theta_{2}<\ldots<\theta_{M-1}<2\pi$. Here $\Gamma_m (\theta , t)$ denotes the total circulation at radius $r_m(\theta , t)$, and $Z_m(\theta , t)$ denotes the parametrization of the spiral at time $t$ with respect to the polar parameter $\theta\in\R$. The spirals were first introduced by Prandtl \cite{prandtl92} as a model of air vortices appearing on the tips of wings and were used to study the effect of the vortices on the lifting force of the wings. 

The authors' previous work \cite{cko,cko1} provide a theory of logarithmic spiral vortex sheets as solutions to the 2D Euler equations, as well as constructs a generic family of nonsymmetric spirals with branches of relative angles $\theta_{m}=\theta_{m}(a)$ arbitrarily close to $m\pi/M$ (i.e. halves of the symmetric angles) as $a\to \infty$, for all $0\le m\le M-1$.

In this paper, we are concerned with the vortex sheets \eqref{log_spiral}, where
\eqnb\label{alex}
g_m= g \quad\text{and}\quad \theta_m = 2\pi m /M \quad \text{for} \ \ m=0,\ldots , M-1,
\eqne
for some $g\in \R \setminus \{ 0 \}$, where $M\geq 3$, which are often referred to as \emph{Alexander spirals} \cite{alexander}.
Such spirals satisfy the $2$D Euler equation on $(0,\infty)\times \R^2$ as well as the Birkhoff-Rott \cite{birkhoff,rott} equations
\begin{equation}\label{B-R}
\partial_t Z_m(t,\Gamma_m)^*=\frac{1}{2\pi i}\mathrm{p.v.}\int\sum_{k=0}^M\frac{\d\Gamma_k}{Z_m(t,\Gamma_m)-Z_k(t,\Gamma_k)}, \quad 0\le m\le M-1,
\end{equation}
if $a^2+1-2\mu + 2a\mu i =-2a^2 g \coth (\pi A/M)$ (see (1.18) in \cite{cko} and \cite{EL}, respectively). Regarding the notation in \eqref{B-R}, we denote by $z^*$ the complex conjugate of $z$ and consider the spirals parametrized using the cumulative vorticity $\Gamma$ rather than the angle of rotation $\theta$. For more information concerning the Birkhoff-Rott equation we refer the reader to \cite[Chapter~9]{majda_bertozzi}.
The last equality above can be satisfied by, for example, first choosing any sufficiently large $a>0$ and then determining uniquely $\mu\in \R$, and $g\in \R\setminus \{ 0 \}$, see the comments below \cite[Corollary~1.4]{cko} for details. Therefore, in this work we suppose that such $a$, $\mu$, $g$ have already been determined.

We consider the Alexander spiral described by $Z_m (t, \Gamma_{m})$ (as in \eqref{log_spiral}--\eqref{alex}), and we focus our attention on perturbations of the form
\eqnb\label{perturbation}
Z_m (t,\Gamma_{m}) + \varepsilon \zeta_m (t,\Gamma_{m}),\quad 0\leq m \leq M-1.
\eqne
It can  be shown that the linearization of the Birkhoff-Rott equation \eqref{B-R} around $Z_m$ takes the form
\eqnb\label{intro_linearization}
t\partial_{t}\zeta_{m}(t,\theta)^{*} - \frac{2\mu -1}{2a}\partial_{\theta}\zeta_{m}(t,\theta)^{*} =-\mathcal{H}_{m}\zeta\,(t,\theta),\quad t>0,
\eqne
for $m=0,\ldots,M-1$, where we denote $\zeta (t,\theta ) := \zeta (t , \Gamma (t, \theta ))$ and we set
\begin{align}\label{eq-op-h}
\HH_{m}\zeta(\theta) \coloneqq \frac{ag}{\pi i} \mathrm{f.p.}\int_{-\infty}^{\infty} \sum_{k=0}^{M-1} \frac{(\zeta_{m}(\theta)-\zeta_{k}(\sigma+\theta+\Delta_{km}))\ee^{2a\sigma}\ee^{-2i\theta}}{(1-\ee^{(a+i)\sigma} \ee^{i\Delta_{km}})^{2}}\,\d\sigma ,
\end{align}
and $\Delta_{km}\coloneqq \theta_k-\theta_m = 2\pi (k-m)/M$,
see Section~\ref{sec_linearized_br} for details. Here ``f.p.'' denotes the \emph{finite part}, which is often referred to as the \emph{Hadamard regularization} \cite{hadamard}, see Definition~\ref{def_fp_curve}. Note that the definition of $\mathcal{H}_m\zeta $ involves two difficulties: one is the problem of convergence of the integral as $\sigma \to \infty$, and another one is the finite part, which is concerned with the singularity at $\sigma =0$. Note that the singularity occurs only in the case $m=k$.

In this work we study the structure of the operator $\HH_m$, and we identify a form of perturbation $\zeta$ which exposes its duality. Namely, for any $0\le m\le M-1$, we consider perturbations of the form
\begin{align}\label{intro_form-zeta} 
\zeta_{m}(t,\theta) \coloneqq X_{m}(t) \zeta^{+}_{m}(t,\theta)\ee^{i\theta}+Y_{m}(t)\zeta^{-}_{m}(t,\theta)\ee^{i\theta},\quad t>0, \, \theta\in\R
\end{align}
where $X_{m},Y_{m}:(0,\infty)\to\C$ are smooth functions and 
\begin{align}\label{intro_eq-a-2}
\zeta_{m}^{\pm}(t,\theta) \coloneqq \ee^{\pm i\alpha\ln\frac{\Gamma_{m}(t,\theta)}{g}} = \ee^{\pm i\alpha(2a(\theta-\theta_{m}) + (2\mu -1)\ln t)},\quad t>0, \ \theta\in\R
\end{align}
are bounded oscillations with the parameter $\alpha >0$. In fact, the form \eqref{intro_eq-a-2} can be thought of as a Fourier mode in variable $\theta $ with frequency $2a\alpha$.

In order to state our main result we first set 
\eqnb\label{def_cs}
\begin{split}
c_{0} \coloneqq \frac{ga(a-i)}{(a+i)^{2}}\coth(\pi A/M),\quad c^{\pm}(\alpha) \coloneqq   \frac{g a^{2}(\pm2i\alpha + 1)}{(a+i)^{2}}\coth(\pi B_{\pm}(\alpha) /M),
\end{split}
\eqne
where
\eqnb\label{def_of_A_Bpm}
A\coloneqq -\frac{2ai}{a+i}\quad\text{and}\quad B_{\pm}\coloneqq B_{\pm}(\alpha)\coloneqq -a(\pm 2i\alpha+1)\frac{1+ai}{1+a^{2}},\quad \alpha>0.
\eqne
\begin{theorem}[Main result]\label{thm_main}
Given $M\geq 3$, if $a>0$, $\alpha \in (0,1/2a) \cup (1/2a, \infty)$ are such that
\eqnb\label{main_cond}
c_0 \text{ does not lie on the line passing through }c^{+}, c^{-}
\eqne
(in the complex plane), then there exists a solution $\zeta$ of the form \eqref{intro_form-zeta} to the linearized Birkhoff-Rott equation \eqref{intro_linearization} such that $\|\zeta (t,\,\cdot\,) \|_{L^\infty } \geq C t^{\delta }$ for all $t>0$, where $C,\delta  >0$ are constants dependent only on the parameters $a,\alpha>0$.

Moreover, \eqref{main_cond} holds for all sufficiently large $a>0$ and all $\alpha \ne 1/2a$  sufficiently close to $1/2a$.
\end{theorem}
We emphasize that Theorem~\ref{thm_main} is valid for {\bf all} Alexander spirals with $M\geq 3$ branches.

We also note that, generally speaking,  vortex sheets are known for their lack of stability. For example, the well-known discontinuous shear flows undergo the Kelvin-Helmholtz instability \cite{birkhoff,helmholtz_68,kelvin} as a solution to the Birkhoff-Rott equation, see \cite[Section~9.3]{majda_bertozzi}. A recent result of Protas et al. \cite{protas} shows that the Prandtl-Munk vortex sheet, which is a more involved object, but still geometrically flat, is also linearly unstable. We also note that such vortex sheet does not satisfy the 2D Euler equations \cite{lns}, unlike logarithmic spirals \eqref{log_spiral} (a result of \cite{cko}).

On the other hand, a recent result of Murray and Wilcox \cite{Mur-Wil} proves that the standard circular discontinuous shear flow does not admit Kelvin-Helmholtz instability, which shows that geometry of a vortex sheet could be related to its stability properties. Thus, considering logarithmic spirals \eqref{log_spiral}, the question of their stability is interesting, as their geometry can be thought of as an interpolation between the  straight and the circular vortex sheets. This shows the relevance of the claim of Theorem~\ref{thm_main}: it not only proves linear instability of logarithmic spiral vortex sheets, but also identifies the geometry of the putative unstable perturbations. We also note a similarity of \eqref{intro_form-zeta} to classical analysis of the instability of the flat vortex sheet (see \cite[p. 368]{majda_bertozzi}).

Finally, we conjecture that the geometric constraint \eqref{main_cond} is satisfied for almost all values of  $a>0$, $\alpha \in (0,1/2a)\cup (1/2a , \infty )$, rather than merely the values guaranteed by the last claim of Theorem~\ref{thm_main}.

A recent paper of Jeong and Said \cite{js}  considers logarithmic spirals from the viewpoint of an initial value problem of the $2$D Euler equations with logarithmic spiral scaling invariance in space, and provides a number of local and global well-posedness results, including initial data of the form of a finite sum of Dirac deltas, which corresponds to logarithmic spiral vortex sheets. The paper also points out (as in \cite[Theorem~A]{EJ}) that 2D flows with bounded vorticity and $m$-fold symmetry with $m\geq 3$ are stable due to the fact that such vorticity function generates a Log-Lipschitz velocity, which in turn induces  relaxation. Theorem~\ref{thm_main} is concerned with unbounded vorticity (recall \cite[(1.21)]{cko} that the total circulation in a disc of radius $R>0$ is $\omega (B(0,R)) \sim R^2$, while the flow is irrotational away from the spiral vortex sheets), and so it demonstrates not only that unbounded vorticity can indeed give rise to (linear) instability,  but also that it occurs for all $m$-fold symmetry regimes with $m\geq 3$.

In order to prove Theorem~\ref{thm_main}, in Section 3, we explore the structure of the operator $\mathcal{H}_m$ defined in \eqref{eq-op-h} above, and observe that the choice of the perturbation $\zeta$ in \eqref{intro_form-zeta} allows us to decouple the linearized Birkhoff-Rott equation \eqref{br-lin} into a system of ODEs.
\begin{proposition}[ODE characterization]\label{prop_ODE_reduction}
Suppose that $X_m(t)$, $Y_m(t)$ ($m\in \{ 0, \ldots ,  M-1\}$) satisfy the perturbation compatibility conditions 
\begin{align}\label{cond-1}
\sum_{k=0}^{M-1} X_{k}(t)\ee^{-i\theta_{k}} = \sum_{k=0}^{M-1} Y_{k}(t)\ee^{-i\theta_{k}} = 0
\end{align}
and
\begin{align}\label{cond-2}
\sum_{k=0}^{M-1} X_{k}(t)\ee^{-2i\theta_{k}} = \sum_{k=0}^{M-1} Y_{k}(t)\ee^{-2i\theta_{k}} = 0
\end{align}
for all $t>0$.
Then the perturbation $\zeta$, given by \eqref{intro_form-zeta}, \eqref{intro_eq-a-2}, satisfies the linearized Birkhoff-Rott equation \eqref{intro_linearization} if and only if the functions $X_m(t)$ and $Y_m(t)$ satisfy 
\eqnb\label{intro_linearization_with_ansatz}
\left\{\begin{aligned}
& t\partial_{t}X_{m}(t)+ i\frac{1-2\mu}{2a}X_{m}(t) = -c_{0}^{*}\,Y_{m}(t)^{*} + g\sum_{k=0}^{M-1} c^{-}_{mk}(\alpha)^{*}\,Y_{k}(t)^{*} , & \\
& t\partial_{t}Y_{m}(t) + i\frac{1-2\mu}{2a}Y_{m}(t) = -c_{0}^{*}\,X_{m}(t)^{*} + g\sum_{k=0}^{M-1} c^{+}_{mk}(\alpha)^{*}\,X_{k}(t)^{*}  && 
\end{aligned}\right.
\eqne
for all $m=0,\ldots , M-1$, $t>0$, where coefficients $c_0$ and $c_{mk}^\pm (\alpha )$ are defined in \eqref{def_cs} and \eqref{def_of_cmk}, respectively.
\end{proposition}
We now show how Theorem~\ref{thm_main} can be proved using Proposition~\ref{prop_ODE_reduction} and a few other facts, which we verify in detail in Sections~\ref{sec_summing}--\ref{sec_imag_part} below.

First, thanks to the ODE characterization provided by Proposition~\ref{prop_ODE_reduction}, we note that the symmetric choice of $X_m$'s and $Y_m$'s, i.e.
\begin{equation*}
X_m(t) =X(t),\quad Y_m (t) = Y(t) \quad \text{for} \ \ m=0,\ldots , M-1,
\end{equation*}
satisfies the perturbation compatibility conditions \eqref{cond-1}--\eqref{cond-2} for any functions $X(t)$, $Y(t)$, and then the ODE system \eqref{intro_linearization_with_ansatz} of $2M$ equations simplifies into $2$ equations,
\eqnb
\left\{\begin{aligned}
& \partial_{t}\left(t^{i\frac{1-2\mu}{2a}}X(t)\right) = -t^{-1+i\frac{1-2\mu}{2a}}(c_{0}^{*} - c^{-}(\alpha)^{*})Y(t)^{*}, \\
& \partial_{t}\left(t^{i\frac{1-2\mu}{2a}}Y(t)\right) = -t^{-1+i\frac{1-2\mu}{2a}}(c_{0}^{*} - c^{+}(\alpha)^{*})X(t)^{*} 
\end{aligned}\right.
\eqne
for $t>0$, which follows immediately from the identity
\eqnb\label{c_pm_sums}
c^{\pm} =g \sum_{k=0}^{M-1} c_{mk}^\pm \quad \text{for} \ \ m=0,\ldots, M-1,
\eqne
which we prove in Section~\ref{sec_summing}.

Thus, since $| {\zeta}_m^{\pm} | =1$ for all $m=0,\ldots,M-1$, the question of growth of $\| \zeta (t,\,\cdot\,) \|_{L^\infty} $ in $t$ reduces to studying the growth of $|X(t)|+|Y(t)|$. Applying the rescaling
\begin{equation*}
\begin{split}
\hat X(t) &\coloneqq \ee^{i\frac{1-2\mu }{2a}t} X(\ee^{t}), \\
\hat Y(t) &\coloneqq \left( \hat X(t) \right)' = -(c_0 - c^- (\alpha ) )^* \left( Y(\ee^t ) \ee^{i\frac{1-2\mu }{2a} t } \right)^* \ee^{i \frac{1-2\mu }{a} t}
\end{split}
\end{equation*}
we obtain 
\eqnb\label{intro_eq-matrix}
\frac{\d }{\d {t}}\begin{pmatrix} \hat X(t) \\[5pt] \hat Y(t) \end{pmatrix} = \begin{pmatrix} 0 & 1 \\[5pt] (c_{0} - c^{-}(\alpha))^*(c_{0} - c^{+}(\alpha)) & i\frac{1-2\mu}{a} \end{pmatrix} \begin{pmatrix} \hat X(t) \\[5pt] \hat Y(t)\end{pmatrix},\quad t\in\R.
\eqne

Thus in order to establish Theorem~\ref{thm_main} it suffices to show that one of the eigenvalues $\lambda_1,\lambda_2 \in \C$ of the matrix in \eqref{intro_eq-matrix} above has positive real part $\delta >0$. Then there exists a constant $C>0$ such that $|\hat X(t)|+|\hat Y(t)| \geq C \ee^{\delta t}$ for all $t\in \R$, and rescaling back we obtain  $| X(t)|+| Y(t)| \geq C t^\delta $ for all $t>0 $, as required.

In order to find the positive real part $\delta$, we note that, since the trace of the matrix in \eqref{intro_eq-matrix} is purely imaginary, the real parts of $\lambda_1,\lambda_2$ have opposite signs, that is
\[
\lambda_1 = u + w_1i,\qquad \lambda_2 = -u+w_2i
\]
for some $u,w_1,w_2\in\R$. Thus it suffices to show that $u\ne 0$. Computing the determinant of the matrix in \eqref{intro_eq-matrix} we obtain
\[
-{(c_0 - c^-)}^* (c_0 - c^+) = \lambda_1 \lambda_2 = (u+w_1i)(-u+w_2i)=-u^2-w_1w_2+u(w_2-w_1)i.
\]
Thus if
\eqnb\label{no_imag_part}
\im \left( {(c_0 - c^-)}^* (c_0 - c^+)  \right) \ne 0
\eqne
then $u\ne 0$, as needed. Clearly, \eqref{no_imag_part} is equivalent to the complex numbers $c_0-c^-$, $c_0-c^+$ not being parallel (as vectors in the complex plane over $\R$), which in turn is equivalent to \eqref{main_cond}.

In order to prove \eqref{main_cond} for $a$, $\alpha$ as in the last claim of Theorem~\ref{thm_main}, we give a geometric argument using asymptotic analysis as $a\to \infty$ and $\alpha \to 1/2a$, see Section~\ref{sec_imag_part} for details.

We note that Theorem~\ref{thm_main} says that growth of the perturbations of order 
$\epsilon$ is only polynomial, which is much slower than the instability rate of the flat shear flow vortex sheet (see \cite[Chapter 9]{majda_bertozzi}).  
We also note that the Kelvin-Helmholtz instability is related with a putative creation of a mixing zone propagating from a vortex sheet, see for example the work of Otto~\cite{otto}, who derived the growth rate in the case of the incompressible porous media flow, as well as subsequent example of Sz\'ekelyhidi Jr \cite{szekelyhidi_cr,szekelyhidi_en} in the case of a flat vortex sheet. We also note a subsequent result of Castro, C\'ordoba and Faraco \cite{ccf}, which construct the first example of a mixing layer arising from a non-flat vortex sheet in the Muskat problem, as well as recent result \cite{men_szeke} of Mengual and Sz\'ekelyhidi Jr, which is concerned with mixing layers for non-flat initial vortex sheets in the $2D$ Euler equations. We emphasize, that in each of these examples the growth rate of the mixing layer is at most polynomial.

One could speculate that this slower rate of growth of instabilities is responsible for the appearance of spiral-like structures in fluids, see \cite{ever} for instance. This makes the open question of nonlinear instability \cite{fsv,lin} very interesting, particularly considering the putative growth rate. Finally, let us  note that there are instances of geometric structures appearing in nature, despite their linear instability, for example the rings of Saturn, whose linear instability was established by Maxwell \cite{max}.

In the following Section~\ref{sec_prelims} we  introduce some preliminary concepts, such as the notion of the \emph{finite part} of an integral, as well as derive the linearized Birkhoff-Rott equation \eqref{intro_linearization}. We then prove the ODE characterization (Proposition~\ref{prop_ODE_reduction}) in Section~\ref{sec_ode_char}, we provide a verification of the identity \eqref{c_pm_sums} in Section~\ref{sec_summing} and we verify the geometric condition \eqref{main_cond} in Section~\ref{sec_imag_part}.

\section{Preliminaries}\label{sec_prelims}

\subsection{The finite part}

Given a smooth function $f:[a,b]\to\R$ and $x\in(a,b)$, we define the \emph{finite part} of $f$ (Hadamard regularization) as
\begin{equation}
\begin{aligned}\label{fp}
\mathrm{f.p.}\int_{a}^{b} \frac{f(t)}{(t-x)^{2}}\,\d t  & \coloneqq  \lim_{\ve\to 0^{+}} \left(\int_{a}^{x-\ve}\frac{f(t)}{(t-x)^{2}}\,\d t +
\int_{x+\ve}^{b}\frac{f(t)}{(t-x)^{2}}\,\d t-\frac{f(x+\ve)+f(x-\ve)}{\ve}\right)
\end{aligned}
\end{equation}
see the classical exposition of Hadamard \cite[Book III, Chapter I]{hadamard} for details. It is known that under the above smoothness assumption the above limits exist and they are related to the principal value
\begin{align}\label{pv}
\mathrm{p.v.}\int_{a}^{b}\frac{f(t)}{t-x}\,\d t = \lim_{\ve\to 0^{+}} \left(\int_{a}^{x-\ve}\frac{f(t)}{t-x}\,\d t +
\int_{x+\ve}^{b}\frac{f(t)}{t-x}\,dt\right),
\end{align}
via a derivative in the sense that
\begin{align*}
\frac{\d}{\d x}\left(\mathrm{p.v.}\int_{a}^{b}\frac{f(t)}{t-x}\,\d t\right) = \mathrm{f.p.}\int_{a}^{b} \frac{f(t)}{(t-x)^{2}}\,\d t,
\end{align*}
which can be proved directly from the fact that the limits \eqref{fp} and \eqref{pv} are uniform with respect to $x$ in compact subsets of $(a,b)$.
\begin{definition}\label{def_fp_curve}
Let $\gamma \colon [a,b]\to\C$ be a smooth curve satisfying $\gamma'(t)\neq 0$ for all $t\in[a,b]$. Given $x\in(a,b)$, we define
\begin{align*}
\mathrm{f.p.}\!\int_{a}^{b} \frac{f(t)\,\d t}{(\gamma(t)-\gamma(x))^{2}}\!\coloneqq\! \mathrm{f.p.}\!\int_{a}^{b} \frac{g(t)\,\d t}{(t-x)^{2}}, \ \ \text{where}
 \ \ g(t)\!\coloneqq\!  f(t)\left(\frac{\gamma(t)-\gamma(x)}{t-x}\right)^{-2}\ \ \text{for} \ \ t\in [a,b].
\end{align*}
\end{definition}
In particular, by \eqref{fp}, we have
\begin{equation}
\begin{aligned}\label{eq-pv-2}
&\mathrm{f.p.}\int_{a}^{b} \!\frac{f(t)\,\d t}{(\gamma(t)-\gamma(x))^{2}} \\
& = \!\!\lim_{\ve\to 0^{+}}\!\! \left(\int_{a}^{x-\ve}\!\!\frac{f(t)\,\d t}{(\gamma(t)\!-\!\gamma(x))^{2}}\!+\!\int_{x+\ve}^{b}\!\!\frac{f(t)\,\d t}{(\gamma(t)-\gamma(x))^{2}}\!-\!\frac{\ve f(x+\ve)}{(\gamma(x+\ve)\!-\!\gamma(x))^{2}} \!-\! \frac{\ve f(x-\ve)}{(\gamma(x-\ve)\!-\!\gamma(x))^{2}}\right).
\end{aligned}
\end{equation}
In the appendix (Section~\ref{appendix}) we provide some more identities related to the notion of the finite part, and we also prove the identity
\begin{align}\label{eq-fp-pv}
\mathrm{f.p.}\int_{a}^{b} \frac{f(t)\,\d t}{(\gamma(t)\!-\!\gamma(x))^{2}} = \mathrm{p.v.}\int_{a}^{b}\frac{\left(f(t)/\gamma'(t)\right)'\d t}{\gamma(t)-\gamma(x)} - \frac{f(b)}{\gamma'(b)(\gamma(b)\!-\!\gamma(x))} + \frac{f(a)}{\gamma'(a)(\gamma(a)\!-\!\gamma(x))},
\end{align}
which is valid for every $x\in (a,b)$, see Proposition~\ref{prop_fp_vs_pv}.

\subsection{Linearized Birkhoff-Rott equation}\label{sec_linearized_br}
Since $Z_m (t,\Gamma_m)$ (as in \eqref{log_spiral}--\eqref{alex})  satisfies the Birkhoff-Rott equation \eqref{B-R}, applying the perturbation \eqref{perturbation}, we obtain
\begin{align*}
\left(\partial_{t}Z_{m}(t,\Gamma_{m}) + \ve\partial_{t}\zeta_{m}(t,\Gamma_{m})\right)^{*} = \frac{1}{2\pi i} \mathrm{p.v.}\!\int\sum_{k=0}^{M-1} \frac{\d\Gamma_{k}}{Z_{m}(t,\Gamma_{m})\!-\!Z_{k}(t,\Gamma_{k}) \!+\! \ve(\zeta_{m}(t,\Gamma_{m})\!-\!\zeta_{k}(t,\Gamma_{k}))}
\end{align*}
for $0\le m\le M-1$. Using the Taylor series expansion
\begin{align*}
\frac{1}{x+\ve y} = \frac{1}{x} \frac{1}{1+\ve y/x} = \frac{1}{x}\sum_{k=0}^{\infty} (-\ve y/x)^{k} = \sum_{k=0}^{\infty} (-\ve)^{k} \frac{y^{k}}{x^{k+1}},
\end{align*}
which is valid for sufficiently small $\ve>0$ and $x\in\R$, $y\in\R\setminus\{0\}$, we obtain
\begin{align*}
\left(\partial_{t}Z_{m}(t,\Gamma_{m}) + \ve\partial_{t}\zeta_{m}(t,\Gamma_{m})\right)^{*} 
& = \frac{1}{2\pi i} \mathrm{p.v.}\int\sum_{k=0}^{M-1} \frac{\d\Gamma_{k}}{Z_{m}(t,\Gamma_{m})-Z_{k}(t,\Gamma_{k})} \\
&\quad -\frac{\ve}{2\pi i} \mathrm{f.p.}\int \sum_{k=0}^{M-1} \frac{\zeta_{m}(t,\Gamma_{m})-\zeta_{k}(t,\Gamma_{k})}{[Z_{m}(t,\Gamma_{m})-Z_{k}(t,\Gamma_{k})]^{2}}\,\d\Gamma_{k} + O(\ve^{2})
\end{align*}
as $\ve\to 0$. Considering only the terms at $\varepsilon$ we obtain a linearization of the Birkhoff-Rott equation,
\begin{align}\label{br-lin}
\partial_{t}\zeta_{m}(t,\Gamma_{m})^{*} = -\frac{1}{2\pi i} \mathrm{f.p.}\int\sum_{k=0}^{M-1} \frac{\zeta_{m}(t,\Gamma_{m})-\zeta_{k}(t,\Gamma_{k})}{[Z_{m}(t,\Gamma_{m})-Z_{k}(t,\Gamma_{k})]^{2}}\,\d\Gamma_{k},\quad 0\le m\le M-1.
\end{align}
We now change the variable $\Gamma_k \mapsto \theta$, by using the parametrization of $Z_{m}(t,\theta )$ and $\Gamma_{m}(t,\theta )$ from \eqref{log_spiral}. To this end, we apply the slight abuse of notation by writing $\zeta_{m}(t,\theta):=\zeta_{m}(t,\Gamma_{m}(t,\theta))$, which after differentiation gives
\begin{align*}
\partial_{\theta}\zeta_{m}(t,\theta) = \partial_{\theta}[\zeta_{m}(t,\Gamma_{m}(t,\theta))] = \partial_{\Gamma_{m}}\zeta_{m}(t,\Gamma_{m}(t,\theta))\partial_{\theta}\Gamma_{m}(t,\theta).
\end{align*}
Thus we obtain
\begin{equation}
\begin{aligned}\label{eq-rr1}
\partial_{t}\zeta_{m}(t,\theta) & = \partial_{t}[\zeta_{m}(t,\Gamma_{m}(t,\theta))] = \partial_{t}\zeta_{m}(t,\Gamma_{m}(t,\theta)) + \partial_{\Gamma_{m}}\zeta_{m}(t,\Gamma_{m}(t,\theta))\partial_{t}\Gamma_{m}(t,\theta) \\
& = \partial_{t}\zeta_{m}(t,\Gamma_{m}(t,\theta)) + \frac{\partial_{\theta}\zeta_{m}(t,\theta)}{\partial_{\theta}\Gamma_{m}(t,\theta)}\partial_{t}\Gamma_{m}(t,\theta).
\end{aligned}
\end{equation}
Combining \eqref{eq-rr1} with the equality
\begin{align*}
\frac{\partial_{t}\Gamma_{m}(t,\theta)}{\partial_{\theta}\Gamma_{m}(t,\theta)} = \frac{g(2\mu -1)t^{2\mu-2}  \ee^{2a (\theta - \theta_{m} )}}{2ag t^{2\mu-1}  \ee^{2a (\theta - \theta_{m} )}} = \frac{2\mu -1}{2a t}
\end{align*}
gives
\begin{align}\label{rr-2}
\partial_{t}\zeta_{m}(t,\Gamma_{m}(t,\theta)) = \partial_{t}\zeta_{m}(t,\theta) - \frac{2\mu -1}{2a t}\partial_{\theta}\zeta_{m}(t,\theta).
\end{align}
Substituting \eqref{rr-2} and \eqref{log_spiral} into \eqref{br-lin} yields the linearized Birkhoff-Rott equation in the new variable 
\begin{equation*}
\begin{aligned}
&\partial_{t}\zeta_{m}(t,\theta)^{*} - \frac{2\mu -1}{2a t}\partial_{\theta}\zeta_{m}(t,\theta)^{*}  =
-\frac{a}{\pi i} \mathrm{f.p.}\int_{-\infty}^{\infty} \sum_{k=0}^{M-1} \frac{\zeta_{m}(t,\theta)-\zeta_{k}(t,\theta')}{[Z_{m}(t,\theta)-Z_{k}(t,\theta')]^{2}}\Gamma_{k}(\theta')\,\d\theta' \\
&\qquad = -\frac{a}{\pi i} \mathrm{f.p.}\int_{-\infty}^{\infty} \sum_{k=0}^{M-1} \frac{\zeta_{m}(t,\theta)-\zeta_{k}(t,\theta')}{(t^\mu \ee^{a (\theta - \theta_{m})} \ee^{i\theta }-t^\mu \ee^{a(\theta' - \theta_{k})} \ee^{i\theta'})^{2}}gt^{2\mu -1 }  \ee^{2a (\theta' - \theta_{k} )}\,\d\theta'\\
&\qquad = -\frac{a}{t\pi i} \mathrm{f.p.}\int_{-\infty}^{\infty} \sum_{k=0}^{M-1} \frac{(\zeta_{m}(t,\theta)-\zeta_{k}(t,\theta'))g\ee^{2a (\theta' - \theta_{k} )}}{(\ee^{a (\theta - \theta_{m})} \ee^{i\theta }-\ee^{a(\theta' - \theta_{k})} \ee^{i\theta'})^{2}}\,\d\theta'.
\end{aligned}
\end{equation*}
Therefore, the formal change of variables $\sigma\coloneqq \theta'-\theta+\theta_{m}-\theta_{k}$ and $ \Delta_{km} \coloneqq \theta_{k}-\theta_{m}$, gives \eqref{intro_linearization}, i.e.
\begin{align*}
t\partial_{t}\zeta_{m}(t,\theta)^{*} - \frac{2\mu -1}{2a}\partial_{\theta}\zeta_{m}(t,\theta)^{*}
=-\mathcal{H}_{m}\zeta\,(t,\theta),\quad t>0, 
\end{align*}
for $m=0\ldots,M-1$, where $\zeta (t,\cdot )=(\zeta_{0}(t,\cdot ),\ldots,\zeta_{M-1}(t,\cdot ))\in L^{\infty}(\R,\C^{M})$ and $\mathcal{H}_m$ is defined in \eqref{eq-op-h}, as required.

\section{Proof of the ODE characterization}\label{sec_ode_char}
In this section we prove Proposition~\ref{prop_ODE_reduction}, namely we show that if the compatibility conditions \eqref{cond-1}--\eqref{cond-2} hold, that is,
\[\begin{aligned}
&\sum_{k=0}^{M-1} X_{k}(t)\ee^{-i\theta_{k}} = \sum_{k=0}^{M-1} Y_{k}(t)\ee^{-i\theta_{k}} = 0,\\
&\sum_{k=0}^{M-1} X_{k}(t)\ee^{-2i\theta_{k}} = \sum_{k=0}^{M-1} Y_{k}(t)\ee^{-2i\theta_{k}} = 0
\end{aligned} \]
for $t>0$, then the perturbation $\zeta$, defined in \eqref{intro_form-zeta}, satisfies the linearized Birkhoff-Rott equation \eqref{intro_linearization} if and only if $X_m(t)$, $Y_m(t)$ ($m=0,\ldots , M-1$) are solutions of the system
\eqnb\label{intro_linearization_with_ansatz_repeat}
\left\{\begin{aligned}
& t\partial_{t}X_m(t)+ i\frac{1-2\mu}{2a}X_m(t) = -c_{0}^{*}\,Y_{m}(t)^{*} + g\sum_{k=0}^{M-1} c^{-}_{mk}(\alpha)^{*}\,Y_{k}(t)^{*},  \\
& t\partial_{t}Y_{m}(t) + i\frac{1-2\mu}{2a}Y_{m}(t) = -c_{0}^{*}\,X_{m}(t)^{*} + g\sum_{k=0}^{M-1} c^{+}_{mk}(\alpha)^{*}\,X_{k}(t)^{*}
\end{aligned}\right.
\eqne
for all $m=0,\ldots , M-1$, $t>0$. In the above equation $c_0= \frac{ga(a-i)}{(a+i)^{2}}\coth(\pi A/M) $ (recall \eqref{def_cs}) and
\begin{equation}\label{def_of_cmk}
\begin{aligned}
c_{mk}^{+}(\alpha) & \coloneqq  \frac{a^{2}(2i\alpha + 1)}{(a+i)^{2}} \frac{\mathcal{B}^{+}_{mk}}{\sinh(\pi B_{+})}, &&\quad \alpha>0,\ \alpha\neq1/2a,\\
c_{mk}^{-}(\alpha) & \coloneqq \frac{a^{2}(-2i\alpha + 1)}{(a+i)^{2}} \frac{\mathcal{B}^{-}_{mk}}{\sinh(\pi B_{-})}, &&\quad \alpha>0,
\end{aligned}
\end{equation}
where $B_\pm$ is defined in \eqref{def_of_A_Bpm} 
and
\eqnb\label{def_Bmk}
\mathcal{B}^{\pm}_{mk}(\alpha)\coloneqq \ee^{\Delta_{km}B_{\pm}}\left\{\begin{aligned}
& \ee^{-\pi B_{\pm}} && \text{if } \ \Delta_{km}>0,\\
& \cosh(\pi B_{\pm}) && \text{if } \ \Delta_{km}=0 ,\\
& \ee^{\pi B_{\pm}} && \text{if } \ \Delta_{km}<0.
\end{aligned}\right.
\eqne
To this end we first study the structure of the operator $\HH$ defined in \eqref{eq-op-h}, where
\[
\HH_{m}\zeta(\theta) = \frac{ag}{\pi i}\,\mathrm{f.p.}\int_{-\infty}^{\infty} \sum_{k=0}^{M-1} \frac{(\zeta_{m}(\theta)-\zeta_{k}(\sigma+\theta+\Delta_{km}))\ee^{2a\sigma}\ee^{-2i\theta}}{(1-\ee^{(a+i)\sigma} \ee^{i\Delta_{km}})^{2}}\,\d\sigma ,
\]
for $m=0,\ldots,M-1$, and we decompose it into
\[
\HH=\KK-\JJ,
\]
where
\begin{align}\label{d-h-1}
\KK_{m}\zeta(\theta)& \coloneqq \frac{ag}{\pi i}\,\mathrm{f.p.}\int_{-\infty}^{\infty} \sum_{k=0}^{M-1} \frac{\zeta_{m}(\theta)\ee^{2a\sigma}\ee^{-2i\theta}}{(1-\ee^{(a+i)\sigma} \ee^{i\Delta_{km}})^{2}}\,\d\sigma, \\ \label{d-h-2}
\JJ_{m}\zeta(\theta) & \coloneqq \frac{ag}{\pi i}\,\mathrm{f.p.}\int_{-\infty}^{\infty} \sum_{k=0}^{M-1} \frac{\zeta_{k}(\sigma+\theta+\Delta_{km})\ee^{2a\sigma}\ee^{-2i\theta}}{(1-\ee^{(a+i)\sigma} \ee^{i\Delta_{km}})^{2}}\,\d\sigma,
\end{align}
for $m=0,\ldots,M-1$ and $\zeta=(\zeta_{0},\ldots,\zeta_{M-1})\in L^{\infty}(\R,\C^{M})$. Then we characterize the diagonal part $\KK   $ and the remaining part $\JJ $. To be precise, we show in Section \ref{sec_H1} that
\eqnb\label{claim_of_sec31}
\KK \zeta = c_0 \zeta \ee^{-2i\theta}\quad\text{for} \ \  \zeta \in L^{\infty}(\R,\C^{M})
\eqne
and, in Section \ref{sec_H2}, we show that, if the compatibility conditions \eqref{cond-1}--\eqref{cond-2} hold, then
\eqnb\label{claim_of_sec_32}
\JJ_m \zeta (t,\theta )= \zeta^{+}_{m}(t,\theta)\ee^{-i\theta}g \sum_{k=0}^{M-1} X_{k}(t) c_{mk}^{+}(\alpha) + \zeta^{-}_{m}(t,\theta)\ee^{-i\theta}g\sum_{k=0}^{M-1} Y_{k}(t) c_{mk}^{-}(\alpha)
\eqne
for $m=0,\ldots , M-1$, $t>0$, provided $\zeta$ is of the form \eqref{intro_form-zeta}, that is,
\[
\zeta_{m}(t,\theta) = X_{m}(t) \zeta^{+}_{m}(t,\theta)\ee^{i\theta}+Y_{m}(t)\zeta^{-}_{m}(t,\theta)\ee^{i\theta}, 
\]
where
\[
\zeta_{m}^{\pm}(t,\theta) = \ee^{\pm i\alpha\ln\frac{\Gamma_{m}(t,\theta)}{g}} = \ee^{\pm i\alpha(2a(\theta-\theta_{m}) + (2\mu -1)\ln t)},
\]
for $t>0$, $\theta\in\R$ and $m=0,\ldots , M-1$.
Note that we have the following lemma.
\begin{lemma}[Linear independence of $\zeta^\pm$]\label{lem_LI}
For each $k\in \{ 0,\ldots , M-1\}$ and  $t>0$, the functions ${\zeta}_k^\pm (t,\,\cdot\,)$ are linearly independent in the sense that if, for some $p,q\in \C$,
\begin{align*}
p\zeta_{k}^{+}(t,\theta) + q\zeta_{k}^{-}(t,\theta) = 0\quad \text{for all} \ \ \theta\in \R,
\end{align*}
then $p=q=0$.
\end{lemma}
\begin{proof}
Differentiating the above assumption in $\theta$ gives $2ia\alpha p \zeta_{k}^{+}(t,\theta) -2ia\alpha  q \zeta_{k}^{-}(t,\theta) = 0$,
which leads to the system of equations
\begin{equation*}
\begin{pmatrix}
\zeta_{k}^{+}(t,\theta) & \zeta_{k}^{-}(t,\theta)\\ \zeta_{k}^{+}(t,\theta)  &
 -\zeta_{k}^{-}(t,\theta) \\
\end{pmatrix}
\begin{pmatrix}
p \\ q
\end{pmatrix} = 0 \quad \text{for all} \ \ \theta \in \R.
\end{equation*}
A direct computation shows that the  matrix of the system is invertible, which proves the claim.
\end{proof}

Thus, in order to prove Proposition~\ref{prop_ODE_reduction}, we first recall \eqref{intro_linearization}, that is
\[
t\partial_{t}\zeta_{m}(t,\theta)^{*} - \frac{2\mu -1}{2a}\partial_{\theta}\zeta_{m}(t,\theta)^{*} =-\mathcal{H}_{m}\zeta\,(t,\theta)
\]
for $m=0,\ldots,M-1$. A direct computation shows that
\[\begin{split}
&t\partial_{t}\zeta_{m}(t,\theta) - \frac{2\mu -1}{2a}\partial_{\theta}\zeta_{m}(t,\theta),  \\
&\qquad = \left(t\partial_{t}X_{m}(t)+ i\frac{1-2\mu}{2a}X_{m}(t)\right)  \zeta^{+}_{m}(t,\theta)\ee^{i\theta}
+ \left(t\partial_{t}Y_{m}(t) + i\frac{1-2\mu}{2a}Y_{m}(t)\right)  \zeta^{-}_{m}(t,\theta)\ee^{i\theta}.
\end{split}\]
Thus, noting that $(\zeta_m^+)^*=\zeta_m^-$ (and vice versa) the above claims \eqref{claim_of_sec31}--\eqref{claim_of_sec_32} show that \eqref{intro_linearization} holds if and only if
\[\begin{split}
&\left(t\partial_{t}X_{m}(t)+ i\frac{1-2\mu}{2a}X_{m}(t)\right)^{*}  \zeta^{-}_{m}(t,\theta)\ee^{-i\theta}
+ \left(t\partial_{t}Y_{m}(t) + i\frac{1-2\mu}{2a}Y_{m}(t)\right)^{*} \zeta^{+}_{m}(t,\theta)\ee^{-i\theta}\\
&\qquad = -( \KK_m \zeta (t,\theta ) - \JJ_m \zeta (t,\theta ) )  \\
&\qquad= -c_0( X_{m}(t) \zeta^{+}_{m}(t,\theta)+Y_{m}(t)\zeta^{-}_{m}(t,\theta)) \ee^{-i\theta } \\
&\qquad\qquad + \zeta^{+}_{m}(t,\theta)\ee^{-i\theta}g \sum_{k=0}^{M-1} X_{k}(t) c_{mk}^{+}(\alpha) + \zeta^{-}_{m}(t,\theta)\ee^{-i\theta}g\sum_{k=0}^{M-1} Y_{k}(t) c_{mk}^{-}(\alpha)
\end{split}
\]
for all $\theta \in \R$. Hence, it remains to notice that $\ee^{-i\theta }$ cancels out, and so, due to the linear independence of $\zeta^\pm$ (Lemma~\ref{lem_LI}), the last equality is equivalent to \eqref{intro_linearization_with_ansatz_repeat}, as required.

\subsection{Operator $\KK $}\label{sec_H1}
As mentioned above, in this section we show that $\KK\zeta = c_0 \zeta \ee^{-2i\theta }$. 

To this end we first verify that the improper integral (as $\sigma \to \infty$) in the definition \eqref{d-h-1} of $\mathcal{K}$ converges, namely that the limit
\begin{align}\label{eq-f-int}
\int_{1}^{\infty} \sum_{k=0}^{M-1} \frac{\ee^{2a\sigma}}{(1-\ee^{(a+i)\sigma} \ee^{i\Delta_{km}})^{2}}\,\d\sigma \coloneqq  \lim_{R\to\infty} \int_{1}^{R} \sum_{k=0}^{M-1} \frac{\ee^{2a\sigma}}{(1-\ee^{(a+i)\sigma} \ee^{i\Delta_{km}})^{2}}\,\d\sigma
\end{align}
exists.

To this end we note the identity
\begin{align}\label{eq-row}
\frac{\ee^{2a\sigma}}{1-\ee^{(a+i)\sigma+i\Delta_{km}}} = -\ee^{(a-i)\sigma-i\Delta_{km}} -\ee^{-2i\sigma-2i\Delta_{km}} + \frac{\ee^{-2i\sigma-2i\Delta_{km}}}{1-\ee^{(a+i)\sigma+i\Delta_{km}}}
\end{align}
(which was observed by \cite[(3.12)]{EL}). On the one hand, dividing \eqref{eq-row} by $1-\ee^{(a+i)\sigma+i\Delta_{km}}$, we obtain
\begin{align}\label{eq-row3}
\frac{\ee^{2a\sigma}}{(1-\ee^{(a+i)\sigma+i\Delta_{km}})^{2}} = -\frac{\ee^{(a-i)\sigma-i\Delta_{km}}}{{1-\ee^{(a+i)\sigma+i\Delta_{km}}}} - \frac{\ee^{-2i\sigma-2i\Delta_{km}}}{{1-\ee^{(a+i)\sigma+i\Delta_{km}}}} + \frac{\ee^{-2i\sigma-2i\Delta_{km}}}{(1-\ee^{(a+i)\sigma+i\Delta_{km}})^{2}}.
\end{align}
On the other hand, multiplying \eqref{eq-row} by $\ee^{-(a+i)\sigma - i\Delta_{km}}$ gives
\begin{align}\label{eq-row2}
\frac{\ee^{(a-i)\sigma-i\Delta_{km}}}{1-\ee^{(a+i)\sigma+i\Delta_{km}}} = -\ee^{-2i\sigma-2i\Delta_{km}} -\ee^{-a\sigma-3i\sigma-3i\Delta_{km}} + \frac{\ee^{-a\sigma -3i\sigma-3i\Delta_{km}}}{1-\ee^{(a+i)\sigma+i\Delta_{km}}}.
\end{align}
Combining \eqref{eq-row3} and \eqref{eq-row2} gives
\begin{equation}
\begin{aligned}\label{e-c-1}
\frac{\ee^{2a\sigma}}{(1-\ee^{(a+i)\sigma} \ee^{i\Delta_{km}})^{2}} & = \ee^{-2i\sigma-2i\Delta_{km}} +\ee^{-a\sigma-3i\sigma-3i\Delta_{km}} - \frac{\ee^{-a\sigma -3i\sigma-3i\Delta_{km}}}{1-\ee^{(a+i)\sigma+i\Delta_{km}}} \\
&\qquad - \frac{\ee^{-2i\sigma-2i\Delta_{km}}}{{1-\ee^{(a+i)\sigma+i\Delta_{km}}}} + \frac{\ee^{-2i\sigma-2i\Delta_{km}}}{(1-\ee^{(a+i)\sigma+i\Delta_{km}})^{2}}.
\end{aligned}
\end{equation}
Since $\sum_{k=0}^{M-1} \ee^{-2i\sigma-2i\Delta_{km}}=0$ (recall that \eqref{alex} holds and $M\ge 3$), the equation \eqref{e-c-1} implies that the limit \eqref{eq-f-int} exists, as required.

In order to prove the claim of this section, we show the following.
\begin{lemma}\label{lem-fp}
For any $0\le m\le M-1$,
\begin{equation}
\begin{aligned}\label{eq-fp-1}
\mathrm{f.p.}\int_{-\infty}^{\infty} \sum_{k=0}^{M-1}\frac{\ee^{2a\sigma}}{(1-\ee^{(a+i)\sigma} \ee^{i\Delta_{km}})^{2}}\,\d\sigma
= \frac{\pi i }{a g} c_0 .
\end{aligned}
\end{equation}
\end{lemma}
\begin{proof}
Given $r>0$, we first integrate by parts by using the relation \eqref{eq-fp-pv} between f.p. and p.v. with $\gamma (\sigma ) \coloneqq \ee^{(a+i)\sigma } \ee^{i\Delta_{km}}$ to obtain
\begin{equation}
\begin{aligned}\label{int-p-1}
&\mathrm{f.p.}\int_{-r}^{r} \sum_{k=0}^{M-1} \frac{\ee^{2a\sigma}}{(1-\ee^{(a+i)\sigma} \ee^{i\Delta_{km}})^{2}}\,\d\sigma \\
&\quad = \frac{1}{a+i}\,\mathrm{f.p.}\int_{-r}^{r} \sum_{k=0}^{M-1} \ee^{(a-i)\sigma}\ee^{-i\Delta_{km}}\frac{\d}{\d\sigma}\frac{1}{1-\ee^{(a+i)\sigma} \ee^{i\Delta_{km}}}\,\d\sigma \\
& \quad = \frac{1}{a+i}\sum_{k=0}^{M-1} \frac{\ee^{(a-i)r}\ee^{-i\Delta_{km}}}{1-\ee^{(a+i)r} \ee^{i\Delta_{km}}}-\frac{1}{a+i}\sum_{k=0}^{M-1} \frac{\ee^{-(a-i)r}\ee^{-i\Delta_{km}}}{1-\ee^{-(a+i)r} \ee^{i\Delta_{km}}} \\
&\qquad - \frac{a-i}{a+i}\mathrm{p.v.}\int_{-r}^{r} \sum_{k=0}^{M-1} \frac{\ee^{(a-i)\sigma}\ee^{-i\Delta_{km}}}{1-\ee^{(a+i)\sigma} \ee^{i\Delta_{km}}}\,\d\sigma.
\end{aligned}
\end{equation}
Since
\begin{align*}
\sum_{k=0}^{M-1} \frac{\ee^{-(a-i)r}\ee^{-i\Delta_{km}}}{1-\ee^{-(a+i)r} \ee^{i\Delta_{km}}} = \sum_{k=0}^{M-1} \frac{\ee^{-i\Delta_{km}}}{\ee^{(a+i)r}- \ee^{i\Delta_{km}}}   \to 0 \quad \text{as} \ \ r\to\infty,
\end{align*}
and, by \eqref{eq-row2},
\begin{align*}
\sum_{k=0}^{M-1} \frac{e^{(a-i)r-i\Delta_{km}}}{1-\ee^{(a+i)r+i\Delta_{km}}} & = -\sum_{k=0}^{M-1} \ee^{-2ir-2i\Delta_{km}} -\sum_{k=0}^{M-1} \ee^{-ar-3ir-3i\Delta_{km}} + \sum_{k=0}^{M-1} \frac{\ee^{-ar -3ir-3i\Delta_{km}}}{1-\ee^{(a+i)r+i\Delta_{km}}} \\
& = -\sum_{k=0}^{M-1} \ee^{-ar-3ir-3i\Delta_{km}} + \sum_{k=0}^{M-1} \frac{\ee^{-ar -3ir-3i\Delta_{km}}}{1-\ee^{(a+i)r+i\Delta_{km}}}\to 0 \quad \text{as} \ \ r\to \infty,
\end{align*}
we take the limit $r\to \infty$ in  \eqref{int-p-1} to obtain
\begin{equation*}
\begin{aligned}
&\mathrm{f.p.}\int_{-\infty}^{\infty} \sum_{k=0}^{M-1} \frac{\ee^{2a\sigma}}{(1-\ee^{(a+i)\sigma} \ee^{i\Delta_{km}})^{2}}\,\d\sigma
 = - \frac{a-i}{a+i}\,\mathrm{p.v.}\int_{-\infty}^{\infty} \sum_{k=0}^{M-1} \frac{\ee^{(a-i)\sigma}\ee^{-i\Delta_{km}}}{1-\ee^{(a+i)\sigma} \ee^{i\Delta_{km}}}\,\d\sigma.
\end{aligned}
\end{equation*}
In order to compute the last principal value integral we note that the poles
\begin{equation*}
\begin{split}\sigma_{j} \coloneqq -(2\pi j+\Delta_{km})\frac{1+ai}{1+a^{2}},\quad j\in\Z
\end{split}
\end{equation*}
(note that the locations of the residua is as in Figure~\ref{fig_contours} below) of each ingredient in the last integrand in \eqref{int-p-1} are all simple, with residua (for each given $k= 0, \ldots , M-1$)
\eqnb\label{res_for_K}
\res_{\sigma =\sigma_j} \frac{\ee^{(a-i)\sigma}\ee^{-i\Delta_{km}}}{1-\ee^{(a+i)\sigma} \ee^{i\Delta_{km}}} = -\frac{\ee^{(a-i)\sigma_{j}} \ee^{-i\Delta_{km}}}{a+i} =-\frac{\ee^{2a\sigma_j}}{a+i} = -\frac{\ee^{(2\pi j + \Delta_{km})A} }{a+i},\quad j\in \Z.
\eqne
Note that, since $\re\, A = -2a/(1+a^{2})<0$ (recall \eqref{def_of_A_Bpm}), we sum all residua in the lower half plane, namely at $\sigma_0, \sigma_1, \ldots $ in the case $\Delta_{km} >0$ and $\sigma_1,\sigma_2,\ldots $ in the case $\Delta_{km}<0$. Thus, for $\Delta_{km} >0$,
\begin{align}\label{conv-1a}
 \sum_{j\ge 0}\ee^{(2\pi j + \Delta_{km}) A}
= \frac{\ee^{\Delta_{km}A}}{1-\ee^{2\pi A}}= \frac{-\ee^{\Delta_{km}A}}{2\sinh(\pi A)} \ee^{-\pi A},
\end{align}
and for $\Delta_{km}<0$ we have
\begin{align}\label{conv-1b}
\sum_{j\ge 1}\ee^{(2\pi j + \Delta_{km}) A}
= \frac{\ee^{\Delta_{km}A}\ee^{2\pi A}}{1-\ee^{2\pi A}} =  \frac{-\ee^{\Delta_{km}A}}{2\sinh(\pi A)} \ee^{\pi A}.
\end{align}
We set
\begin{align*}
d_j \coloneqq \left| \frac{\sigma_j + \sigma_{j+1}}2 \right|,\quad j\in\Z
\end{align*}
and we consider the the sequence of contours
\begin{equation*}
\Gamma_{j} \coloneqq \{d_{j} \ee^{i\varphi} \ | \ \varphi\in [\pi,2\pi]\} \quad \text{for} \ \ j\ge 1,
\end{equation*}
see Figure~\ref{fig_contours} below. Since
\begin{equation}\label{conv-gam}
\int_{\Gamma_{j}}\sum_{k=0}^{M-1} \frac{\ee^{(a-i)\sigma}\ee^{-i\Delta_{km}}}{1-\ee^{(a+i)\sigma} \ee^{i\Delta_{km}}}\,\d\sigma\to 0,\quad \text{ as }j\to\infty,
\end{equation}
 which can be calculated directly (see the proof of \cite[Section 7]{cko}, for example), the residue theorem, \eqref{res_for_K}, \eqref{conv-1a}, \eqref{conv-1b} and \eqref{conv-gam} give
\begin{align}\label{eq-pv-1}
\mathrm{p.v.}\int_{-\infty}^{\infty} \sum_{k=0}^{M-1} \frac{\ee^{(a-i)\sigma}\ee^{-i\Delta_{km}}}{1-\ee^{(a+i)\sigma} \ee^{i\Delta_{km}}}\,\d\sigma  = -\frac{1}{a+i} \frac{\pi i}{\sinh(\pi A)}\sum_{k=0}^{M-1}\mathcal{A}_{mk},
\end{align}
where we define
\begin{equation*}
\mathcal{A}_{mk}\coloneqq \ee^{\Delta_{km}A}
\left\{\begin{aligned}
& \ee^{-\pi A}, && \quad k>m, \\
& \cosh(\pi A), && \quad k=m, \\
& \ee^{\pi A},  && \quad k<m.
\end{aligned}\right.
\end{equation*}
Observe that in the above definition, we obtained $\cosh (\pi A)$ from the fact that the contour of integration passes through the simple pole $\sigma_0$ (and so gives the average of the contour integrals for contours avoiding $\sigma_0$ from either side). Combining \eqref{int-p-1} with \eqref{eq-pv-1} gives
\begin{align*}
\mathrm{f.p.}\int_{-\infty}^{\infty} \sum_{k=0}^{M-1}\frac{\ee^{2a\sigma}}{(1-\ee^{(a+i)\sigma} \ee^{i\Delta_{km}})^{2}}\,\d\sigma
= \frac{a-i}{(a+i)^{2}}\frac{\pi i}{\sinh(\pi A)}\sum_{k=0}^{M-1}\mathcal{A}_{k} = \pi i\frac{a-i}{(a+i)^{2}}\coth(\pi A/M),
\end{align*}
where in the last equality, we summed the finite geometric series (as in \cite[Appendix]{cko}). This together with \eqref{int-p-1} gives \eqref{eq-fp-1}, as required.
\end{proof}

\subsection{Operator $\JJ$}\label{sec_H2}
Here we study the properties of the operator $\JJ$, given by \eqref{d-h-2}, that is,
\[
\JJ_{m}\zeta  \coloneqq \frac{ag}{\pi i} \,\mathrm{f.p.}\int_{-\infty}^{\infty} \sum_{k=0}^{M-1} \frac{\zeta_{k}(\sigma+\theta+\Delta_{km})\ee^{2a\sigma}\ee^{-2i\theta}}{(1-\ee^{(a+i)\sigma} \ee^{i\Delta_{km}})^{2}}\,\d\sigma
\]
for $0\le m\le M-1$, where $\zeta\in L^{\infty}(\R,\C^{M})$.

As mentioned above, the main result of this section is the following.
\begin{proposition}[Structure of $\JJ_m\zeta $]\label{prop_str_J}
If $c_{mk}^\pm (\alpha )$ are defined by \eqref{def_of_cmk} and $X_k$, $Y_k$ ($k=0,\ldots , M-1$) satisfy the perturbation compatibility conditions \eqref{cond-1}--\eqref{cond-2} then
\[
\JJ_m \zeta\,(t,\theta) = \zeta^{+}_{m}(t,\theta)\ee^{-i\theta}g \sum_{k=0}^{M-1} X_{k}(t) c_{mk}^{+}(\alpha) + \zeta^{-}_{m}(t,\theta)\ee^{-i\theta}g\sum_{k=0}^{M-1} Y_{k}(t) c_{mk}^{-}(\alpha),
\]
where $\zeta$ is the perturbation of the form \eqref{intro_form-zeta}, namely,
\[
\zeta_{m}(t,\theta) = X_{m}(t) \zeta^{+}_{m}(t,\theta)\ee^{i\theta}+Y_{m}(t)\zeta^{-}_{m}(t,\theta)\ee^{i\theta}, 
\]
where
\[
\zeta_{m}^{\pm}(t,\theta) =  \ee^{\pm i\alpha\ln\frac{\Gamma_{m}(t,\theta)}{g}} = \ee^{\pm i\alpha(2a(\theta-\theta_{m})+ (2\mu -1)\ln t)}
\]
for $t>0$, $\theta\in\R$ and $m=0,\ldots , M-1$.
\end{proposition}
In the proof we first show (in Section~\ref{sec_fp_to_pv}) that $\JJ$ can be transformed into the following form
\eqnb\label{fp_to_pv} \begin{split}
\JJ_m \zeta &= -\zeta^{+}_{m}(t,\theta)\ee^{-i\theta} \frac{a^{2}g}{\pi i}\frac{2i\alpha + 1}{a+i}\,\mathrm{p.v.}\int_{-\infty}^{\infty} \sum_{k=0}^{M-1} X_{k}(t) \frac{\ee^{(2i\alpha+1)a\sigma}}{1-\ee^{(a+i)\sigma} \ee^{i\Delta_{km}}}\,\d\sigma \\
& \quad - \zeta^{-}_{m}(t,\theta)\ee^{-i\theta} \frac{a^{2}g}{\pi i}\frac{-2i\alpha + 1}{a+i}\,\mathrm{p.v.}\int_{-\infty}^{\infty} \sum_{k=0}^{M-1} Y_{k}(t)\frac{\ee^{(-2i\alpha+1)a\sigma}}{1-\ee^{(a+i)\sigma} \ee^{i\Delta_{km}}}\,\d\sigma
\end{split}
\eqne
for $0\le m \le M-1$.

In Section~\ref{sec_res_thm} below, we will then use the Residue Theorem to deduce Proposition~\ref{prop_str_J} from \eqref{fp_to_pv}, namely that
\begin{align}
&\mathrm{p.v.}\int_{-\infty}^{\infty}  \sum_{k=0}^{M-1} X_{k}(t)  \frac{\ee^{( 2i\alpha+1)a\sigma}}{1-\ee^{(a+i)\sigma} \ee^{i\Delta_{km}}}\,\d\sigma  =  -\frac{\pi i}{a^{2}} \frac{a+i}{2i\alpha + 1} \sum_{k=0}^{M-1} X_{k}(t) c_{mk}^+, \label{cmk_relation1} \\[3pt]
&\mathrm{p.v.}\int_{-\infty}^{\infty}  \sum_{k=0}^{M-1} Y_{k}(t)  \frac{\ee^{(- 2i\alpha+1)a\sigma}}{1-\ee^{(a+i)\sigma} \ee^{i\Delta_{km}}}\,\d\sigma  =  -\frac{\pi i}{a^{2}} \frac{a+i}{-2i\alpha + 1} \sum_{k=0}^{M-1} Y_{k}(t) c_{mk}^-,  \label{cmk_relation2}
\end{align}
given compatibility conditions \eqref{cond-1}--\eqref{cond-2} hold. (Recall $c_{mk}^\pm$'s are defined in \eqref{def_of_cmk}.)
In order to obtain \eqref{cmk_relation1}--\eqref{cmk_relation2}  we will use contour integration over both half circles in $\{ \im >0 \}$ and $\{ \im <0 \} \subset \C$, see Figure~\ref{fig_contours}.
\begin{center}
\includegraphics[width=0.5\textwidth]{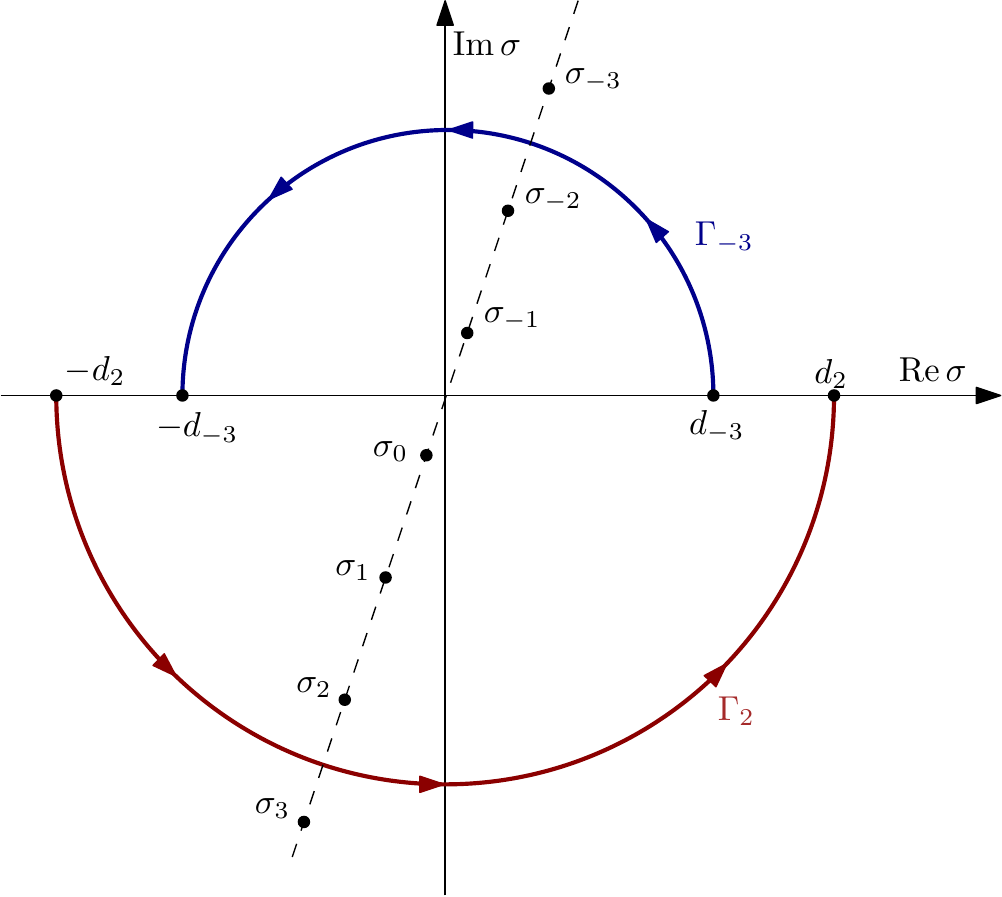}
\captionof{figure}{A sketch of the contour integration in \eqref{cmk_relation1}--\eqref{cmk_relation2}.}\label{fig_contours}
\end{center}
In order to obtain \eqref{cmk_relation1}--\eqref{cmk_relation2}, we first observe the following, provided the perturbation compatibility conditions \eqref{cond-1}--\eqref{cond-2} hold.\\

\noindent\emph{Case 1.} \eqref{cmk_relation1}, $\alpha \in (0,1/2a)$ - we show that
\begin{align}\label{eq-lim-p1}
\lim_{j\to\infty}\int_{\Gamma_{j}} \sum_{k=0}^{M-1} X_{k}(t) \frac{\ee^{(2i\alpha+1)a\sigma}}{1-\ee^{(a+i)\sigma} \ee^{i\Delta_{km}}}\,\d\sigma = 0.
\end{align}
\noindent\emph{Case 2.} \eqref{cmk_relation1}, $\alpha \in (1/2a,\infty )$ - we show that
\begin{align}\label{eq-lim-p2}
\lim_{j\to-\infty}\int_{\Gamma_{j}} \sum_{k=0}^{M-1} X_{k}(t) \frac{\ee^{(2i\alpha+1)a\sigma}}{1-\ee^{(a+i)\sigma} \ee^{i\Delta_{km}}}\,\d\sigma = 0.
\end{align}
\noindent\emph{Case 3.} \eqref{cmk_relation2}, $\alpha \in (0,\infty )$ - we show that
\begin{align}\label{eq-lim-p3}
\lim_{j\to\infty}\int_{\Gamma_{j}} \sum_{k=0}^{M-1} Y_{k}(t) \frac{\ee^{(-2i\alpha+1)a\sigma}}{1-\ee^{(a+i)\sigma} \ee^{i\Delta_{km}}}\,\d\sigma = 0.
\end{align}
see Section~\ref{sec_contours} for details, where $\Gamma_{j}$ denotes the upper of lower semicircle, depending on the sign of $j$, of radius $d_j$ which is oriented counterclockwise, that is
\[
\Gamma_j \coloneqq \begin{cases} \{d_j \ee^{i\varphi } \ | \ \varphi \in [0,\pi] \},\quad &j\le -1,\\
\{d_j \ee^{i\varphi } \ | \ \varphi \in [\pi , 2\pi ] \},\quad &j\geq 0 .
\end{cases}
\]
We note that the choice between the two limits $j\to \pm \infty $ in \eqref{eq-lim-p1}--\eqref{eq-lim-p3} is made based on the summability of the residua at one half-plane or the other ($\{ \im>0 \}$ or $\{ \im <0 \}$). For example, for \eqref{eq-lim-p2}, we have
\[
\ee^{(2i\alpha +1) a\sigma_j }=\ee^{(2\pi j + \Delta_{km} ) B_+}
\]
(which is analogous to the observation \eqref{res_for_K} above, where we obtained $A$ instead of $B_+$), and so, since $\re \,B_+=-a(1-2a\alpha )/(1+a^2)>0$ for $\alpha >1/2a$, the sum of the residua at $\sigma_j$'s will only be summable as $j\to -\infty$, which is also, heuristically speaking, the reason the contour integral in \eqref{eq-lim-p2} converges (which we verify in detail in Section~\ref{sec_contours}).

One can then use the Residue Theorem with respect to either the top or the bottom half-disk (see Figure~\ref{fig_contours}), depending on the case above, to obtain \eqref{cmk_relation1}--\eqref{cmk_relation2}, which we describe in detail in Section~\ref{sec_res_thm} below.

\subsubsection{Finite part to principal value}\label{sec_fp_to_pv}
In this section we show \eqref{fp_to_pv}. To this end, similarly as in the case of $\KK_m \zeta$ (recall \eqref{eq-f-int}), we first observe that the improper integral in the definition \eqref{d-h-2} of $\JJ_m \zeta$ converges. Indeed, since the map $\zeta$ is bounded for $\theta\in\R$, it suffices to verify  convergence of the improper integral
\begin{align}\label{conv-int-1}
\int_{1}^{\infty} \sum_{k=0}^{M-1} \frac{\zeta_{k}(\sigma+\theta+\Delta_{km})\ee^{2a\sigma}\ee^{-2i\theta}}{(1-\ee^{(a+i)\sigma} \ee^{i\Delta_{km}})^{2}}\,\d\sigma,
\end{align}
which in turn, by \eqref{e-c-1}, holds provided
\begin{align}\label{conv-int-2}
I(\sigma)\coloneqq \sum_{k=0}^{M-1} \zeta_{k}(\sigma+\theta+\Delta_{km})\ee^{-2i\sigma-2i\Delta_{km}} \ee^{-2i\theta} = 0
\end{align}
for $\sigma\in(1,\infty)$. To prove \eqref{conv-int-2}, we recall the definition \eqref{intro_form-zeta} of $\zeta$ and observe that
\begin{equation}
\begin{aligned}\label{eq-a-1}
 \zeta^{\pm}_{k}(t,\sigma+\theta+\Delta_{km})\ee^{i(\sigma+\theta+\Delta_{km})} = \ee^{\pm 2ia\alpha\sigma}  \ee^{i(\sigma+\Delta_{km})} \zeta^{\pm}_{m}(t,\theta)\ee^{i\theta},
\end{aligned}
\end{equation}
which shows that
\begin{align*}
I(\sigma) = \ee^{(2a\alpha - 1)i\sigma}\ee^{-i\theta}\zeta^{+}_{m}(t,\theta)\sum_{k=0}^{M-1} X_{k}(t)\ee^{-i\Delta_{km}}
 + \ee^{-(2a\alpha+1)i\sigma}\ee^{-i\theta}\zeta^{-}_{m}(t,\theta)\sum_{k=0}^{M-1} Y_{k}(t)\ee^{-i\Delta_{km}}=0,
\end{align*}
due to the perturbation compatibility condition \eqref{cond-1}, as desired.

Due to \eqref{d-h-2} and the relation \eqref{eq-fp-pv} between f.p. and p.v, for any $r>0$, we obtain
\begin{equation}
\begin{aligned}\label{eq-part-diff-1}
\mathcal{J}_{m}\zeta  
& = \frac{ag}{\pi i}\frac{1}{a+i}\mathrm{f.p.}\int_{-r}^{r} \sum_{k=0}^{M-1} \zeta_{k}(\sigma+\theta+\Delta_{km})\ee^{(a-i)\sigma}\ee^{-i\Delta_{km}}\ee^{-2i\theta}\frac{\d}{\d\sigma}\frac{1}{1-\ee^{(a+i)\sigma} \ee^{i\Delta_{km}}}\,\d\sigma \\
& \quad = -\frac{ag}{\pi i}\frac{1}{a+i}\mathrm{p.v.}\int_{-r}^{r} \sum_{k=0}^{M-1} \frac{\partial_{\sigma}\zeta_{k}(\sigma+\theta+\Delta_{km})\ee^{(a-i)\sigma}\ee^{-i\Delta_{km}}\ee^{-2i\theta}}{1-\ee^{(a+i)\sigma} \ee^{i\Delta_{km}}}\,\d\sigma \\
&\qquad - \frac{ag}{\pi i}\frac{a-i}{a+i}\mathrm{p.v.}\int_{-r}^{r} \sum_{k=0}^{M-1} \frac{\zeta_{k}(\sigma+\theta+\Delta_{km})\ee^{(a-i)\sigma}\ee^{-i\Delta_{km}}\ee^{-2i\theta}}{1-\ee^{(a+i)\sigma} \ee^{i\Delta_{km}}}\,\d\sigma \\
&\qquad +\frac{ag}{\pi i}\frac{1}{a+i}\sum_{k=0}^{M-1}\frac{\zeta_{k}(r+\theta+\Delta_{km})\ee^{(a-i)r}\ee^{-i\Delta_{km}}\ee^{-2i\theta}}{1-\ee^{(a+i)r} \ee^{i\Delta_{km}}} \\
&\qquad - \frac{ag}{\pi i}\frac{1}{a+i}\sum_{k=0}^{M-1} \frac{\zeta_{k}(-r+\theta+\Delta_{km})\ee^{-(a-i)r}\ee^{-i\Delta_{km}}\ee^{-2i\theta}}{1-\ee^{-(a+i)r} \ee^{i\Delta_{km}}}.
\end{aligned}
\end{equation}
Since $\zeta_{k}$ is a bounded function, it is not difficult to check that
\begin{align*}
\sum_{k=0}^{M-1} \frac{\zeta_{k}(-r+\theta+\Delta_{km})\ee^{-(a-i)r}\ee^{-i\Delta_{km}}\ee^{-2i\theta}}{1-\ee^{-(a+i)r} \ee^{i\Delta_{km}}} \to 0 \quad \text{as} \ \ r\to\infty.
\end{align*}
On the other hand, by the equality \eqref{eq-row2} and \eqref{conv-int-2}, we have
\begin{align*}
&\sum_{k=0}^{M-1}\frac{\zeta_{k}(r+\theta+\Delta_{km})\ee^{(a-i)r}\ee^{-i\Delta_{km}}\ee^{-2i\theta}}{1-\ee^{(a+i)r} \ee^{i\Delta_{km}}} = -\sum_{k=0}^{M-1}\zeta_{k}(r+\theta+\Delta_{km})\ee^{-2i\sigma-2i\Delta_{km}}\ee^{-2i\theta} \\
& \qquad + \sum_{k=0}^{M-1}\zeta_{k}(r+\theta+\Delta_{km})\ee^{-2i\theta}\left(-\ee^{-a\sigma-3i\sigma-3i\Delta_{km}} + \frac{\ee^{-a\sigma -3i\sigma-3i\Delta_{km}}}{1-e^{(a+i)\sigma+i\Delta_{km}}}\right)\\
&\quad = \sum_{k=0}^{M-1}\zeta_{k}(r+\theta+\Delta_{km})\ee^{-2i\theta}\left(-\ee^{-a\sigma-3i\sigma-3i\Delta_{km}} + \frac{\ee^{-a\sigma -3i\sigma-3i\Delta_{km}}}{1-\ee^{(a+i)\sigma+i\Delta_{km}}}\right)\to 0 \ \  \text{as} \ r\to\infty.
\end{align*}
Therefore, passing in \eqref{eq-part-diff-1} with $r\to\infty$, we have
\begin{equation}
\begin{aligned}\label{eq-part-diff}
\mathcal{J}_{m}\zeta & = -\frac{ag}{\pi i}\frac{1}{a+i}\mathrm{p.v.}\int_{-\infty}^{\infty} \sum_{k=0}^{M-1} \frac{\partial_{\sigma}\zeta_{k}(\sigma+\theta+\Delta_{km})\ee^{(a-i)\sigma}\ee^{-2i\theta}\ee^{-i\Delta_{km}}}{1-\ee^{(a+i)\sigma} \ee^{i\Delta_{km}}}\,\d\sigma \\
&\quad - \frac{ag}{\pi i}\frac{a-i}{a+i}\mathrm{p.v.}\int_{-\infty}^{\infty} \sum_{k=0}^{M-1} \frac{\zeta_{k}(\sigma+\theta+\Delta_{km})\ee^{(a-i)\sigma}\ee^{-2i\theta}\ee^{-i\Delta_{km}}}{1-\ee^{(a+i)\sigma} \ee^{i\Delta_{km}}}\,\d\sigma.
\end{aligned}
\end{equation}
In order to find $\p_\sigma \zeta_k$ we recall the definition \eqref{intro_eq-a-2} to note that
\begin{equation}
\begin{aligned}\label{eq-a-3}
& \partial_{\sigma}\left( \zeta^{\pm}_{k}(t,\sigma+\theta+\Delta_{km}) \ee^{i (\sigma+\theta+\Delta_{km})}\right) \\
& \qquad = \ee^{i(\sigma + \Delta_{km}+\theta)} (\pm2ia\alpha \ee^{\pm2ia\alpha\sigma} \zeta^{\pm}_{m}(t,\theta)+i\ee^{\pm2ia\alpha\sigma} \zeta^{\pm}_{m}(t,\theta))  \\
& \qquad = i \ee^{i(\sigma + \Delta_{km}+\theta)} \ee^{\pm2ia\alpha\sigma} (\pm2a\alpha +1 )\zeta^{\pm}_{m}(t,\theta).
\end{aligned}
\end{equation}
Applying \eqref{eq-a-3} and \eqref{eq-a-1} in \eqref{eq-part-diff} gives
\begin{align*}
\JJ_m \zeta&= -\frac{a^{2}g}{\pi i}\frac{2i\alpha + 1}{a+i}\,\mathrm{p.v.}\int_{-\infty}^{\infty} \sum_{k=0}^{M-1} X_{k}(t)\ee^{2ia\alpha\sigma}\frac{\zeta^{+}_{m}(t,\theta)\ee^{a\sigma}\ee^{-i\theta}}{1-\ee^{(a+i)\sigma} \ee^{i\Delta_{km}}}\,\d\sigma \\
& \qquad - \frac{a^{2}g}{\pi i}\frac{-2i\alpha + 1}{a+i}\,\mathrm{p.v.}\int_{-\infty}^{\infty} \sum_{k=0}^{M-1} Y_{k}(t) \ee^{-2ia\alpha\sigma}\frac{  \zeta^{-}_{m}(t,\theta)\ee^{a\sigma}\ee^{-i\theta}}{1-\ee^{(a+i)\sigma} \ee^{i\Delta_{km}}}\,\d\sigma,
\end{align*}
as required.

\subsubsection{Contour integrals}\label{sec_contours}
Here we assume the perturbation compatibility conditions \eqref{cond-1}--\eqref{cond-2} and we verify the convergence claims \eqref{eq-lim-p1}--\eqref{eq-lim-p3} of integrals along  semicircles $\Gamma_j$.\\

\begin{lemma}\label{lem-c}
There is a positive constant $c>0$ and $j_{0}\ge 1$ such that, for any $0\le k,m \le M-1$ and $|j|\ge j_{0}$, we have
\begin{align*}
\left|\ee^{(a+i)\sigma + i\Delta_{km}} - 1\right|\ge c,\quad \sigma\in\Gamma_{j}.
\end{align*}
\end{lemma}
\begin{proof}
Let us observe that, for $j,l\le -1$ and $\varphi\in[0,2\pi]$, we have
\begin{align*}
|d_{j}\ee^{i\varphi} - \sigma_{l}| \ge |d_{j} - |\sigma_{l}|| \ge \min\{d_{j}-|\sigma_{j+1}|, |\sigma_{j}|-d_{j}\} = \frac{\pi}{(1+a^{2})^{1/2}}.
\end{align*}
On the other hand, if we assume that $j,l\ge 1$ and $\varphi\in[0,2\pi]$, then
\begin{align*}
|d_{j}\ee^{i\varphi} - \sigma_{l}| \ge |d_{j} - |\sigma_{l}|| \ge \min\{d_{j}-|\sigma_{j}|, |\sigma_{j+1}|-d_{j}\} = \frac{\pi}{(1+a^{2})^{1/2}}.
\end{align*}
Hence, if we take $\ve\coloneqq \frac{\pi}{(1+a^{2})^{1/2}}$, then there is $j_{0}\ge 1$ such that for any $|j|\ge j_{0}$ and $l\in\Z$, we have
\begin{align*}
|d_{j}\ee^{i\varphi} - \sigma_{l}| \ge \ve,\quad \varphi\in[0,2\pi].
\end{align*}
From \cite[Lemma 2.1]{cko} it follows that for $\ve_{0}\coloneqq \ve(1+a^{2})^{1/2}$ there is $c>0$ such that
\begin{align}\label{eq-l-1}
|\ee^{z} - 1| \ge c,\quad z\in \C\setminus\bigcup_{l\in\Z} B(2l\pi i,\ve_{0}).
\end{align}
Then, for any $\sigma\in\Gamma_{j}$, we have
\begin{align*}
\sigma\in\Gamma_{j}\subset\left(\bigcup_{l\in\Z}B(\sigma_{l},\ve)\right)^{c} = \left(\bigcup_{l\in\Z} B\left(\frac{-i(2l\pi +\Delta_{km})}{a+i},\ve\right)\right)^{c},
\end{align*}
which gives
\begin{align*}
(a+i)\sigma + i\Delta_{km}\in \left(\bigcup_{l\in\Z} B\left(2l\pi i,\ve_{0}\right)\right)^{c}.
\end{align*}
Then the assertion of the lemma follows from the inequality \eqref{eq-l-1} and the proof is completed.
\end{proof}

In order to show \eqref{eq-lim-p1}, i.e. that
\eqnb\label{eq-lim-p1_repeat}
\lim_{j\to\infty}\int_{\Gamma_{j}} \sum_{k=0}^{M-1} X_{k}(t) \frac{\ee^{(2i\alpha+1)a\sigma}}{1-\ee^{(a+i)\sigma} \ee^{i\Delta_{km}}}\,\d\sigma = 0,
\eqne
we first set
 \begin{align*}
\gamma_{j}(\varphi)& \coloneqq  d_j (2i\alpha+1) \ee^{i\varphi} = d_j (2i\alpha+1) (\cos\varphi+i\sin\varphi) \\
& = d_{j}(-2\alpha\sin\varphi + \cos\varphi+2i\alpha\cos\varphi + i\sin\varphi)
\end{align*}
Let $\varphi_{1},\varphi_{2}\in(\pi,3\pi/2)$ be such that $\varphi_{1}<\varphi_{2}$ and
\begin{align*}
\cot\varphi_{1} = 1/a,\quad \cot\varphi_{2} = 2\alpha
\end{align*}
Let us take $\varphi_{0}\in (\varphi_{1},\varphi_{2})$ and choose $\varphi_{3}\in(\pi,\varphi_{1})$ such that
\begin{align}\label{est-i-1}
\re \,\gamma_j (\varphi ) & = d_j(-2\alpha\sin\varphi + \cos\varphi) \le -d_j/2,\quad \varphi\in [\pi,\varphi_{3}].
\end{align}
Then, for any $\varphi\in[\varphi_{3},\varphi_{0}]$, we have $-\sin\varphi \ge c_{1} > 0$ and
\begin{equation}
\begin{aligned}\label{est-i-2}
\re \,\gamma_j (\varphi ) & = d_j(-2\alpha\sin\varphi + \cos\varphi) =  -d_j\sin\varphi(2\alpha - \cot\varphi) \\
&= -d_j\sin\varphi(\cot\varphi_{2} - \cot\varphi) \le -d_j\sin\varphi(\cot\varphi_{2} - \cot\varphi_{0}) \\
&\le d_jc_{1}(\cot\varphi_{2} - \cot\varphi_{0}).
\end{aligned}
\end{equation}
Let us decompose the contour on the sum $\Gamma_{j} = \Gamma_{j,1}\cup\Gamma_{j,2}$, where
\begin{align}\label{dec-1}
\Gamma_{j,1}  \coloneqq \{d_{j} \ee^{i\varphi} \ | \ \varphi\in [\pi,\varphi_{0}]\} \quad\text{and}\quad
\Gamma_{j,2}  \coloneqq \{d_{j} \ee^{i\varphi} \ | \ \varphi\in [\varphi_{0},2\pi]\}.
\end{align}
Then, using Lemma \ref{lem-c}, we have the estimates
\begin{align*}
\left|\int_{\Gamma_{j,1}} \sum_{k=0}^{M-1} X_{k}(t) \frac{\ee^{(2i\alpha+1)a\sigma}\,\d\sigma}{1-\ee^{(a+i)\sigma} \ee^{i\Delta_{km}}} \right|
& \le \int_{\pi}^{\varphi_{0}} \sum_{k=0}^{M-1} |X_{k}(t)|\frac{\ee^{a\re\,\gamma_{j}(\varphi)}|\gamma_{j}'(\varphi)|\,\d\varphi}{|1-\ee^{(a+i)\gamma_{j}(\varphi)+i\Delta_{km}}|} \\
& \le \int_{\pi}^{\varphi_{0}} \sum_{k=0}^{M-1} |X_{k}(t)|c^{-1}d_{j}\ee^{a\re\,\gamma_{j}(\varphi)},
\end{align*}
which together with \eqref{est-i-1} and \eqref{est-i-2} implies that
\begin{align}\label{eq-lim-1}
\lim_{j\to\infty} \int_{\Gamma_{j,1}} \sum_{k=0}^{M-1} X_{k}(t)\frac{\ee^{(2i\alpha+1)a\sigma}\,\d\sigma}{1-\ee^{(a+i)\sigma} \ee^{i\Delta_{km}}} = 0.
\end{align}
Let us observe that, from \eqref{eq-row} we have
\begin{align*}
\frac{\ee^{(2i\alpha+1)a\sigma}}{1-\ee^{(a+i)\sigma+i\Delta_{km}}} = -\ee^{2ia\alpha\sigma-i\sigma-i\Delta_{km}} -\ee^{2ia\alpha\sigma-a\sigma-2i\sigma-2i\Delta_{km}} + \frac{\ee^{(2i\alpha-1)a\sigma-2i\sigma-2i\Delta_{km}}}{1-\ee^{(a+i)\sigma+i\Delta_{km}}},
\end{align*}
which together with the conditions \eqref{cond-1} and \eqref{cond-2} gives
\begin{align}\label{eq-22bb}
\int_{\Gamma_{j,2}} \sum_{k=0}^{M-1} X_{k}(t) \frac{\ee^{(2i\alpha+1)a\sigma}\,\d\sigma}{1-\ee^{(a+i)\sigma} \ee^{i\Delta_{km}}} = \int_{\Gamma_{j,2}} \sum_{k=0}^{M-1} X_{k}(t) \frac{\ee^{(2i\alpha-1)a\sigma-2i\sigma-2i\Delta_{km}}}{1-\ee^{(a+i)\sigma+i\Delta_{km}}}\,\d\sigma.
\end{align}
Let us define
\begin{equation}\label{eq-gam-1}
\begin{aligned}
\hat \gamma_{j}(\varphi)& \coloneqq  d_j (2i\alpha-1)a\ee^{i\varphi} - 2id_{j}\ee^{i\varphi} = d_j (2i\alpha-1) a(\cos\varphi+i\sin\varphi) - 2id_{j}(\cos\varphi+i\sin\varphi)\\
& = d_{j}(-2\alpha a\sin\varphi - a\cos\varphi+2ia\alpha\cos\varphi - ia\sin\varphi) - 2id_{j}\cos\varphi+2d_{j}\sin\varphi
\end{aligned}
\end{equation}
and choose $\varphi_{4} \in (3\pi/2, 2\pi)$ such that
\begin{align}\label{est-i-3}
\re \,\hat\gamma_{j}(\varphi) & = d_{j}(-2\alpha a\sin\varphi - a\cos\varphi + 2\sin\varphi) \le -d_{j}a/2,\quad \varphi\in[\varphi_{4},2\pi].
\end{align}
Let $c_{1}>0$ be such that $-\sin\varphi \ge c_{1} > 0$ for $\varphi\in[\varphi_{0},\varphi_{4}]$. Since $2a\alpha<1$, for $\varphi\in[\varphi_{0},\varphi_{4}]$ we have
\begin{equation}\label{est-i-4}
\begin{aligned}
\re \,\hat\gamma_{j}(\varphi) & = d_{j}(-2\alpha a\sin\varphi - a\cos\varphi + 2\sin\varphi) \le d_{j}(- a\cos\varphi + \sin\varphi) \\
& \le -d_{j}a\sin\varphi(\cot\varphi - 1/a) = -d_{j}a\sin\varphi(\cot\varphi - \cot\varphi_{1}) \\
&\le -d_{j}a\sin\varphi(\cot\varphi_{0} - \cot\varphi_{1}) \le d_{j}a c_1(\cot\varphi_{0} - \cot\varphi_{1}).
\end{aligned}
\end{equation}
By Lemma \ref{lem-c} and the equality \eqref{eq-22bb}, we have the estimates
\begin{align*}
&\left|\int_{\Gamma_{j,2}} \sum_{k=0}^{M-1} X_{k}(t) \frac{\ee^{(2i\alpha+1)a\sigma}\,\d\sigma}{1-\ee^{(a+i)\sigma} \ee^{i\Delta_{km}}} \right|
= \left|\int_{\Gamma_{j,2}} \sum_{k=0}^{M-1} X_{k}(t) \frac{\ee^{(2i\alpha-1)a\sigma-2i\sigma-2i\Delta_{km}}}{1-\ee^{(a+i)\sigma+i\Delta_{km}}}\,\d\sigma\right|\\
&\qquad \le \int_{\varphi_{0}}^{2\pi} \sum_{k=0}^{M-1} |X_{k}(t)|\frac{\ee^{\re\,\hat\gamma_{j}(\varphi)}|\hat\gamma_{j}'(\varphi)|\,\d\varphi}{|1-\ee^{(a+i)\hat\gamma_{j}(\varphi)+i\Delta_{km}}|} \le \int_{\varphi_{0}}^{2\pi} \sum_{k=0}^{M-1} |X_{k}(t)|c^{-1}d_{j}\ee^{\re\,\hat\gamma_{j}(\varphi)} \d \varphi,
\end{align*}
which together with \eqref{est-i-3} and \eqref{est-i-4} gives the limit
\begin{align}\label{eq-lim-2}
\lim_{j\to\infty} \int_{\Gamma_{j,2}} \sum_{k=0}^{M-1} X_{k}(t) \frac{\ee^{(2i\alpha+1)a\sigma}\,\d\sigma}{1-\ee^{(a+i)\sigma} \ee^{i\Delta_{km}}} = 0.
\end{align}
Combining \eqref{eq-lim-1} and \eqref{eq-lim-2} gives \eqref{eq-lim-p1_repeat} as desired.\\

As for \eqref{eq-lim-p2}, i.e. that
\begin{align}\label{eq-lim-p2_repeat}
\lim_{j\to-\infty}\int_{\Gamma_{j}} \sum_{k=0}^{M-1} X_{k}(t) \frac{\ee^{(2i\alpha+1)a\sigma}}{1-\ee^{(a+i)\sigma} \ee^{i\Delta_{km}}}\,\d\sigma = 0,
\end{align}
we let $\varphi_{1},\varphi_{2}\in(0,\pi/2)$ be such that $\varphi_{1}>\varphi_{2}$ and
\begin{align*}
\cot\varphi_{1} = 1/a,\qquad \cot\varphi_{2} = 2\alpha.
\end{align*}
Let us take $\varphi_{0}\in (\varphi_{2},\varphi_{1})$ and choose $\varphi_{3}\in(\pi/2,\pi)$ such that
\begin{align}\label{est-i-5}
\re \,\gamma_j (\varphi) & = d_j(-2\alpha\sin\varphi + \cos\varphi) \le -d_j/2,\quad \varphi\in[\varphi_{3},\pi].
\end{align}
Then, for any $\varphi\in[\varphi_{0},\varphi_{3}]$, we have $\sin\varphi \ge c_{1}>0$ and
\begin{equation}\label{est-i6a}
\begin{aligned}
\re \,\gamma_j (\varphi ) & = d_j(-2\alpha\sin\varphi + \cos\varphi) =  d_j\sin\varphi(\cot\varphi-2\alpha) \\
&= d_j\sin\varphi(\cot\varphi-\cot\varphi_{2}) \le d_j\sin\varphi(\cot\varphi_{0} - \cot\varphi_{2})\\
&\le d_jc_{1}(\cot\varphi_{0} - \cot\varphi_{2}).
\end{aligned}
\end{equation}
Let us decompose the contour on the sum $\Gamma_{j} = \Gamma_{j,1}\cup\Gamma_{j,2}$, where
\begin{align*}
\Gamma_{j,1}  \coloneqq \{d_{j} \ee^{i\varphi} \ | \ \varphi\in [0,\varphi_{0}]\} \quad\text{and}\quad
\Gamma_{j,2}  \coloneqq \{d_{j} \ee^{i\varphi} \ | \ \varphi\in [\varphi_{0},\pi]\}.
\end{align*}
Then, by Lemma \ref{lem-c}, we have the estimates
\begin{align*}
\left|\int_{\Gamma_{j,2}} \sum_{k=0}^{M-1} X_{k}(t) \frac{\ee^{(2i\alpha+1)a\sigma}\,\d\sigma}{1-\ee^{(a+i)\sigma} \ee^{i\Delta_{km}}} \right|
& \le \int_{\varphi_{0}}^{\pi} \sum_{k=0}^{M-1} |X_{k}(t)|\frac{\ee^{a\re\,\gamma_{j}(\varphi)}|\gamma_{j}'(\varphi)|\,\d\varphi}{|1-\ee^{(a+i)\gamma_{j}(\varphi)+i\Delta_{km}}|} \\
& \le \int_{\varphi_{0}}^{\pi} \sum_{k=0}^{M-1} |X_{k}(t)|c^{-1}d_{j}\ee^{a\re\,\gamma_{j}(\varphi)}\d \varphi,
\end{align*}
which together with \eqref{est-i-5} and \eqref{est-i6a} implies that
\begin{align}\label{eq-lim-3b}
\lim_{j\to-\infty} \int_{\Gamma_{j,2}} \sum_{k=0}^{M-1} X_{k}(t)\frac{\ee^{(2i\alpha+1)a\sigma}\,\d\sigma}{1-\ee^{(a+i)\sigma} \ee^{i\Delta_{km}}} = 0.
\end{align}
Then there is $\varphi_{4} \in (0,\pi/2)$ such that
\begin{align}\label{est-i-5ab}
\re \,\hat \gamma_{j}(\varphi) & = d_{j}(-2\alpha a\sin\varphi - a\cos\varphi + 2\sin\varphi) \le -d_{j}a/2,\quad \varphi\in[0,\varphi_{4}],
\end{align}
where $\hat \gamma_{j}$ is defined by \eqref{eq-gam-1}. Let $c_1>0$ be such that $\sin\varphi \ge c_{1} > 0$ for $\varphi\in[\varphi_{4},\varphi_{0}]$. Since $2a\alpha>1$, for any $\varphi\in[\varphi_{4},\varphi_{0}]$, we have
\begin{equation}\label{est-i-5aab}
\begin{aligned}
\re \,\hat\gamma_j (\varphi) &=  d_{j}(-2\alpha a\sin\varphi - a\cos\varphi + 2\sin\varphi) \le d_{j}(- a\cos\varphi + \sin\varphi) \\
& = d_{j}a\sin\varphi(-\cot\varphi + 1/a) = d_{j}a\sin\varphi(-\cot\varphi + \cot\varphi_{1}) \\
& \le d_{j}a\sin\varphi(-\cot\varphi_{0} + \cot\varphi_{1}) \le d_{j}a c_{1}(-\cot\varphi_{0} + \cot\varphi_{1}).
\end{aligned}
\end{equation}
By Lemma \ref{lem-c} and the equality \eqref{eq-22bb} we have
\begin{align*}
&\left|\int_{\Gamma_{j,1}} \sum_{k=0}^{M-1} X_{k}(t) \frac{\ee^{(2i\alpha+1)a\sigma}\,\d\sigma}{1-\ee^{(a+i)\sigma} \ee^{i\Delta_{km}}} \right|
= \left|\int_{\Gamma_{j,1}} \sum_{k=0}^{M-1} X_{k}(t) \frac{\ee^{(2i\alpha-1)a\sigma-2i\sigma-2i\Delta_{km}}}{1-\ee^{(a+i)\sigma+i\Delta_{km}}}\,\d\sigma\right|\\
&\qquad \le \int_{0}^{\varphi_{0}} \sum_{k=0}^{M-1} |X_{k}(t)|\frac{\ee^{\re\,\hat\gamma_{j}(\varphi)}|\hat\gamma_{j}'(\varphi)|\,\d\varphi}{|1-\ee^{(a+i) \hat\gamma_{j}(\varphi)+i\Delta_{km}}|} \le \int_{0}^{\varphi_{0}} \sum_{k=0}^{M-1} |X_{k}(t)|c^{-1}d_{j}\ee^{\re\,\hat\gamma_{j}(\varphi)},
\end{align*}
which together with \eqref{est-i-5ab} and \eqref{est-i-5aab} gives 
\begin{align}\label{eq-lim-4b}
\lim_{j\to-\infty} \int_{\Gamma_{j,1}} \sum_{k=0}^{M-1} X_{k}(t) \frac{\ee^{(2i\alpha+1)a\sigma}\,\d\sigma}{1-\ee^{(a+i)\sigma} \ee^{i\Delta_{km}}} = 0.
\end{align}
Combining \eqref{eq-lim-3b} and \eqref{eq-lim-4b} gives \eqref{eq-lim-p2_repeat}, as required.\\

As for showing \eqref{eq-lim-p3}, i.e. that
\begin{align}\label{eq-lim-p3_repeat}
\lim_{j\to\infty}\int_{\Gamma_{j}} \sum_{k=0}^{M-1} Y_{k}(t) \frac{\ee^{(-2i\alpha+1)a\sigma}}{1-\ee^{(a+i)\sigma} \ee^{i\Delta_{km}}}\,\d\sigma = 0,
\end{align}
we first set
\begin{align*}
\tilde\gamma_{j}(\varphi) & \coloneqq  d_j (-2i\alpha+1) \ee^{i\varphi} = d_j (-2i\alpha+1) (\cos\varphi+i\sin\varphi) \\
& = d_{j}(2\alpha\sin\varphi + \cos\varphi-2i\alpha\cos\varphi + i\sin\varphi).
\end{align*}
Let us take $\ve\in(0,\alpha)$ and choose $3\pi/2 <\varphi_{0}<2\pi$ such that
\begin{align*}
\cot\varphi_{0} = -2\alpha + \ve.
\end{align*}
Let $\varphi_{1}\in(\pi,3\pi/2)$ such that
\begin{align}\label{eq-pp1}
\re \, \tilde\gamma_j (\varphi ) = d_j(2\alpha\sin\varphi + \cos\varphi) \le -d_{j}/2,\quad \varphi\in[\pi,\varphi_{1}].
\end{align}
We observe that, for $\varphi\in[\varphi_{1},\varphi_{0}]$ we have $-\sin\varphi\ge c_{1}>0$ and
\begin{equation}
\begin{aligned}\label{eq-pp2}
\re \,\tilde\gamma_j (\varphi ) & = d_j(2\alpha\sin\varphi + \cos\varphi) = d_j\sin\varphi(2\alpha + \cot\varphi) \\
&\le d_j\sin\varphi(2\alpha + \cot\varphi_{0}) = d_j \ve\sin\varphi \le -d_j c_{1} \ve.
\end{aligned}
\end{equation}
Let us recall the contour decomposition $\Gamma_{j} = \Gamma_{j,1}\cup\Gamma_{j,2}$, where the components are given by \eqref{dec-1}.
Then, using Lemma \ref{lem-c}, we obtain the estimates
\[\begin{split}
\left|\int_{\Gamma_{j,1}} \frac{\ee^{(-2i\alpha+1)a\sigma}\,\d\sigma}{1-\ee^{(a+i)\sigma+i\Delta_k}} \right| & \le  \int_{\pi}^{\varphi_{0}} \frac{\ee^{a\re\, \tilde\gamma_{j}(\varphi)}| \tilde\gamma_{j}'(\varphi)|\,\d\varphi}{|1-\ee^{(a+i) \tilde\gamma_{j}(\varphi)+i\Delta_k}|}\le  c^{-1}d_j \int_{\pi}^{\varphi_{0}}\ee^{a\re\, \tilde\gamma_{j}(\varphi)}\,\d\varphi,
\end{split}
\]
which together with \eqref{eq-pp1} and \eqref{eq-pp2} gives the limit
\begin{align}\label{eq-lim-4ab}
\lim_{j\to-\infty} \int_{\Gamma_{j,1}} \sum_{k=0}^{M-1} Y_{k}(t) \frac{\ee^{(-2i\alpha+1)a\sigma}\,\d\sigma}{1-\ee^{(a+i)\sigma} \ee^{i\Delta_{km}}} = 0.
\end{align}
From \eqref{eq-row} it follows that
\begin{align*}
\frac{\ee^{(-2i\alpha+1)a\sigma}}{1-\ee^{(a+i)\sigma+i\Delta_{km}}} = -\ee^{-2ia\alpha\sigma-i\sigma-i\Delta_{km}} -\ee^{-2ia\alpha\sigma-a\sigma-2i\sigma-2i\Delta_{km}} + \frac{\ee^{(-2i\alpha-1)a\sigma-2i\sigma-2i\Delta_{km}}}{1-\ee^{(a+i)\sigma+i\Delta_{km}}},
\end{align*}
which together with \eqref{cond-2} gives
\begin{align}\label{eq-ab-11}
\int_{\Gamma_{j,2}} \sum_{k=0}^{M-1} Y_{k}(t) \frac{\ee^{(-2i\alpha+1)a\sigma}\,\d\sigma}{1-\ee^{(a+i)\sigma} \ee^{i\Delta_{km}}} = \int_{\Gamma_{j,2}} \sum_{k=0}^{M-1} Y_{k}(t) \frac{\ee^{(-2i\alpha-1)a\sigma-2i\sigma-2i\Delta_{km}}}{1-\ee^{(a+i)\sigma+i\Delta_{km}}}\,\d\sigma.
\end{align}
Let us define
\begin{align*}
\bar\gamma_{j}(\varphi)& \coloneqq  d_j (-2i\alpha-1)a\ee^{i\varphi} - 2id_{j}\ee^{i\varphi} = d_j (-2i\alpha-1) a(\cos\varphi+i\sin\varphi) - 2id_{j}(\cos\varphi+i\sin\varphi)\\
& = d_{j}(2\alpha a\sin\varphi - a\cos\varphi-2ia\alpha\cos\varphi - ia\sin\varphi) - 2id_{j}\cos\varphi+2d_{j}\sin\varphi.
\end{align*}
Let $\varphi_{2}\in(\varphi_{0},2\pi)$ be such that
\begin{align}\label{eq-pp3}
\re\, \bar\gamma_{j}(\varphi) & = d_{j}(2\alpha a\sin\varphi - a\cos\varphi + 2\sin\varphi) \le -d_{j}a/2,\quad \varphi\in[\varphi_{2},2\pi].
\end{align}
We observe that, for $\varphi\in[\varphi_{0},\varphi_{2}]$ we have $-\sin\varphi\ge c_{1}>0$ and
\begin{equation}
\begin{aligned}\label{eq-pp4}
\re\,\bar\gamma_{j}(\varphi) & = d_{j}(2\alpha a\sin\varphi - a\cos\varphi + 2\sin\varphi) = d_{j}a\sin\varphi(2\alpha  - \cot\varphi + 2/a) \\
& \le d_{j}a\sin\varphi(4\alpha - \ve + 2/a) \le d_{j}a\sin\varphi(3\alpha + 2/a) \le -d_{j}ac_{1}(3\alpha + 2/a).
\end{aligned}
\end{equation}
By Lemma \ref{lem-c} and the equality \eqref{eq-ab-11} we have the estimates
\begin{align*}
&\left|\int_{\Gamma_{j,2}} \sum_{k=0}^{M-1} Y_{k}(t)\frac{\ee^{(-2i\alpha+1)a\sigma}\,\d\sigma}{1-\ee^{(a+i)\sigma} \ee^{i\Delta_{km}}}\right| = \left|\int_{\Gamma_{j,2}} \sum_{k=0}^{M-1} Y_{k}(t)\frac{\ee^{(-2i\alpha-1)a\sigma-2i\sigma-2i\Delta_{km}}}{1-\ee^{(a+i)\sigma+i\Delta_{km}}}\,\d\sigma\right|\\
&\qquad \le \int_{\varphi_{0}}^{2\pi} \sum_{k=0}^{M-1} |Y_{k}(t)|\frac{\ee^{\re\,\bar\gamma_{j}(\varphi)}|\bar\gamma_{j}'(\varphi)|\,\d\varphi}{|1-\ee^{(a+i)\bar\gamma_{j}(\varphi)+i\Delta_{km}}|} \le \int_{\varphi_{0}}^{2\pi} \sum_{k=0}^{M-1} |Y_{k}(t)|c^{-1}d_{j}\ee^{\re\,\bar\gamma_{j}(\varphi)},
\end{align*}
which together with \eqref{eq-pp3} and \eqref{eq-pp4} gives the limit
\begin{align}\label{eq-lim-5ab}
\lim_{j\to-\infty} \int_{\Gamma_{j,2}} \sum_{k=0}^{M-1} Y_{k}(t) \frac{\ee^{(-2i\alpha+1)a\sigma}\,\d\sigma}{1-\ee^{(a+i)\sigma} \ee^{i\Delta_{km}}} = 0.
\end{align}
Combining \eqref{eq-lim-4ab} and \eqref{eq-lim-5ab} gives \eqref{eq-lim-p3_repeat}, as required.

\subsubsection{The Residue Theorem}\label{sec_res_thm}
As mentioned below \eqref{eq-lim-p3}, here we prove \eqref{cmk_relation1}--\eqref{cmk_relation2}, assuming the perturbation compatibility conditions \eqref{cond-1}--\eqref{cond-2}. To this end, let us recall that
\begin{equation}\label{eq-b-p}
B_{+} =-a(2i\alpha+1)\frac{1+ai}{1+a^{2}} = -\frac{a}{1+a^{2}}(1-2a\alpha + i(a+2\alpha)),
\end{equation}
and
\begin{equation}\label{eq-b-p2}
B_{-}=-a(-2i\alpha+1)\frac{1+ai}{1+a^{2}} = -\frac{a}{1+a^{2}}(1+2a\alpha + i(a-2\alpha)).
\end{equation}
As for \eqref{cmk_relation1}, case $\alpha <1/2a$, the Residue Theorem applied to the bottom half-disk (recall Fig.~\ref{fig_contours}) gives that
\begin{align*}
\mathrm{p.v.}\int_{-d_{j}}^{d_{j}}\frac{\ee^{(2i\alpha+1)a\sigma}}{1-\ee^{(a+i)\sigma} \ee^{i\Delta_{km}}}\,\,\d\sigma = \int_{\Gamma_{j}}\frac{\ee^{(2i\alpha+1)a\sigma}}{1-\ee^{(a+i)\sigma} \ee^{i\Delta_{km}}}\,\d\sigma + \frac{2\pi i}{a+i}\mathcal{R}_{mk}^{j},\quad j\ge 1.
\end{align*}
where we set
\[
\mathcal{R}_{mk}^{j}\coloneqq \begin{cases} \sum_{l=0}^{j}\ee^{a(2i\alpha+1)\sigma_{l}}, \qquad &  \Delta_{km} > 0, \\[3pt]
 \frac{1}{2}+\sum_{l=1}^{j}\ee^{a(2i\alpha+1)\sigma_{l}}, &  \Delta_{km} = 0,\\[3pt]
\sum_{l=1}^{j}\ee^{a(2i\alpha+1)\sigma_{l}}, & \Delta_{km} < 0.
\end{cases}
\]
Applying the above equality for each $k=0,\ldots , M-1$, multiplying by $X_k$ and summing in $k$, as well as using the vanishing \eqref{eq-lim-p1} of the first term on the right hand side as $j\to \infty$, we obtain
\begin{equation}
\begin{aligned}\label{eq-rr-1}
& \mathrm{p.v.}\int_{-\infty}^{\infty} \sum_{k=0}^{M-1} X_{k}(t) \frac{\ee^{(2i\alpha+1)a\sigma}}{1-\ee^{(a+i)\sigma} \ee^{i\Delta_{km}}}\,\d\sigma  = \lim_{j\to\infty}2\pi i\sum_{k=0}^{M-1} X_{k}(t) \frac{\mathcal{R}_{mk}^{j}}{a+i}.
\end{aligned}
\end{equation}
Since $\alpha <1/2a$, from \eqref{eq-b-p} we have $\mathrm{Re}\,B_{+}<0$ and hence, for each $k,m\in \{ 0 ,\ldots , M-1 \}$
\begin{equation}\label{eq-rr-2}
\lim_{j\to\infty}\mathcal{R}_{mk}^{j} = \sum_{l\ge 0}\ee^{(2\pi l + \Delta_{km}) B_{+}}
=\frac{\ee^{\Delta_{km}B_{+}}}{1-\ee^{2\pi B_{+}}} = -\ee^{\Delta_{km}B_{+}} \frac{\ee^{-\pi B_{+}}}{2\sinh(\pi B_{+})}
 = -\frac{\mathcal{B}^{+}_{mk}}{2\sinh(\pi B_{+})},
\end{equation}
if $\Delta_{km} \in (0,2\pi)$, and similarly
\begin{equation}\label{eq-rr-3}
\lim_{j\to\infty}\mathcal{R}_{mk}^{j}
= \sum_{l\ge 1}\ee^{(2\pi l + \Delta_{km}) B_{+}}
= \frac{\ee^{\Delta_{km}B_{+}}\ee^{2\pi B_{+}}}{1-\ee^{2\pi B_{+}}} = -\ee^{\Delta_{km}B_{+}} \frac{\ee^{\pi B_{+}}}{2\sinh(\pi B_{+})}
= -\frac{\mathcal{B}^{+}_{mk}}{2\sinh(\pi B_{+})},
\end{equation}
if $\Delta_{km} \in (-2\pi,0)$, which also covers the case $\Delta_{km}=0$ after adding $1/2$. (Recall \eqref{def_of_A_Bpm} for the definition of $B_\pm$.) Combining \eqref{eq-rr-1}, \eqref{eq-rr-2}, \eqref{eq-rr-3} gives
\begin{align*}
\mathrm{p.v.}\int_{-\infty}^{\infty}  \sum_{k=0}^{M-1} X_{k}(t)  \frac{\ee^{( 2i\alpha+1)a\sigma}}{1-\ee^{(a+i)\sigma} \ee^{i\Delta_{km}}}\,\d\sigma =  -\frac{1}{a+i}\frac{\pi i}{\sinh(\pi B_{+})} \sum_{k=0}^{M-1} X_{k}(t) \mathcal{B}^{+}_{mk},
\end{align*}
which by \eqref{def_of_cmk} proves \eqref{cmk_relation1}, as required.

In order to verify \eqref{cmk_relation1} in the case $\alpha >1/2a$, we apply the Residue Theorem to the upper half-disk to obtain
\begin{align}\label{eq-v-1}
\frac{1}{2\pi i}\mathrm{p.v.}\int_{-d_{j}}^{d_{j}}\frac{\ee^{(2i\alpha+1)a\sigma}}{1-\ee^{(a+i)\sigma} \ee^{i\Delta_{km}}}\,\d\sigma = - \frac{1}{2\pi i}\int_{\Gamma_{j}}\frac{\ee^{(2i\alpha+1)a\sigma}}{1-\ee^{(a+i)\sigma} \ee^{i\Delta_{km}}}\,\d\sigma - \frac{\mathcal{S}_{mk}^{j}}{a+i},\quad j\le -1,
\end{align}
where we set
\[
\mathcal{S}_{mk}^{j}\coloneqq \begin{cases} \sum_{l=-j}^{-1}\ee^{a(2i\alpha+1)\sigma_{l}},\qquad & \Delta_{km} > 0, \\[3pt]
\frac{1}{2}+\sum_{l=-j}^{-1}\ee^{a(2i\alpha+1)\sigma_{l}}, & \Delta_{km} = 0, \\[3pt]
  \sum_{l=-j}^{0}\ee^{a(2i\alpha+1)\sigma_{l}},  & \Delta_{km} < 0.
  \end{cases}
  \]
Note that in the case $\alpha >1/2a$, the inequality \eqref{eq-b-p} gives $\mathrm{Re}\,B_{+}>0$ and hence
\begin{align}\label{eq-mm-2}
\lim_{j\to-\infty}\mathcal{S}_{mk}^{j} &= \sum_{l\le -1}\ee^{(2\pi l + \Delta_{km}) B_{+}}
= \sum_{l\ge 1}\ee^{(-2\pi l + \Delta_{km}) B_{+}} = \frac{\ee^{\Delta_{km}B_{+}}\ee^{-2\pi B_{+}}}{1-\ee^{-2\pi B_{+}}} = \frac{\ee^{\Delta_{km}B_{+}} \ee^{-\pi B_{+}}}{2\sinh(\pi B_{+})},
\end{align}
if $\Delta_{km} \in (0,2\pi)$ (which also covers the case $\Delta_{km}=0$, as above), and
\begin{align}\label{eq-mm-3}
\lim_{j\to-\infty}\mathcal{S}_{mk}^{j} = \sum_{j\le 0}\ee^{(2\pi j + \Delta_{km}) B_{+}}
=\sum_{j\ge 0}\ee^{(-2\pi j + \Delta_{km}) B_{+}} = \frac{\ee^{\Delta_{km}B_{+}}}{1-\ee^{-2\pi B_{+}}} = \frac{\ee^{\Delta_{km}B_{+}}\ee^{\pi B_{+}}}{2\sinh(\pi B_{+})},
\end{align}
if $\Delta_{km} \in (-2\pi,0)$. Summing \eqref{eq-v-1} in $k$ after multiplication by $X_{k}(t)$ we can apply \eqref{eq-mm-2}, \eqref{eq-mm-3} together with the vanishing \eqref{eq-lim-p2} of the contour integral over $\Gamma_j$, to obtain
\begin{align*}
\mathrm{p.v.}\int_{-\infty}^{\infty} \sum_{k=0}^{M-1} X_{k}(t) \frac{\ee^{(2i\alpha+1)a\sigma}}{1-\ee^{(a+i)\sigma} \ee^{i\Delta_{km}}}\,\d\sigma
 = -\frac{1}{a+i}\frac{\pi i}{\sinh(\pi B_{+})}\sum_{k=0}^{M-1} X_{k}(t)\mathcal{B}^{+}_{mk}
\end{align*}
which again by \eqref{def_of_cmk} gives \eqref{cmk_relation1}, as required.

Finally, in order to show \eqref{cmk_relation2} we observe that, by \eqref{eq-b-p2}, we have $\mathrm{Re}\,B_{-}(\alpha)<0$ for every $\alpha >0$. Hence applying the Residue Theorem to the bottom half-disk, gives
\begin{align}\label{case3}
\mathrm{p.v.}\int_{-d_{j}}^{d_{j}}\frac{\ee^{(-2i\alpha+1)a\sigma}}{1-\ee^{(a+i)\sigma} \ee^{i\Delta_{km}}}\,\d\sigma = \int_{\Gamma_{j}}\frac{\ee^{(-2i\alpha+1)a\sigma}}{1-\ee^{(a+i)\sigma} \ee^{i\Delta_{km}}}\,\d\sigma + \frac{2\pi i}{a+i}\mathcal{T}_{mk}^{j},\quad j\ge 1,
\end{align}
where we set
\[
\mathcal{T}_{mk}^{j}\coloneqq \begin{cases}\sum_{l=0}^{j}\ee^{a(-2i\alpha+1)\sigma_{l}} \qquad & \Delta_{km} > 0,\\[3pt]
 \frac{1}{2}+\sum_{l=1}^{j}\ee^{a(-2i\alpha+1)\sigma_{l}} & \Delta_{km} = 0,\\[3pt]
 \sum_{l=1}^{j}\ee^{a(-2i\alpha+1)\sigma_{l}} & \Delta_{km} < 0.
 \end{cases}
 \]
Note that
\begin{align*}
\lim_{j\to\infty}\mathcal{T}_{mk}^{j} = \sum_{j\ge 0}\ee^{(2\pi j + \Delta_{km}) B_{-}}
 = \frac{\ee^{\Delta_{km}B_{-}}}{1-\ee^{2\pi B_{-}}} = -\ee^{\Delta_{km}B_{-}} \frac{\ee^{-\pi B_{-}}}{2\sinh(\pi B_{-})}
\end{align*}
if $\Delta_{km} \in (0,2\pi)$ and
\begin{align*}
\lim_{j\to\infty}\mathcal{T}_{mk}^{j} = \sum_{j\ge 1}\ee^{(2\pi j + \Delta_{km}) B_{-}}
= \frac{\ee^{\Delta_{km}B_{-}}\ee^{2\pi B_{-}}}{1-\ee^{2\pi B_{-}}} = -\ee^{\Delta_{km}B_{-}} \frac{\ee^{\pi B_{-}}}{2\sinh(\pi B_{-})}
\end{align*}
 if $\Delta_{km} \in (-2\pi,0)$ (which, as above, also covers the case $\Delta_{km}=0$). Then we can multiply \eqref{case3} by $Y_k$, sum in $k=0,\ldots , M-1$, take the limit and use the vanishing \eqref{eq-lim-p3} of the contour integral over $\Gamma_{j}$ to obtain
\begin{align*}
\mathrm{p.v.}\int_{-\infty}^{\infty} \sum_{k=0}^{M-1} Y_{k}(t) \frac{\ee^{(-2i\alpha+1)a\sigma}}{1-\ee^{(a+i)\sigma} \ee^{i\Delta_{km}}}\,\d\sigma  = -\frac{1}{a+i}\frac{\pi i}{\sinh(\pi B_{-})}\sum_{k=0}^{M-1} Y_{k}(t) \mathcal{B}^{-}_{mk}.
\end{align*}
This together with \eqref{def_of_cmk} shows \eqref{cmk_relation2} as required.

\section{Summing the coefficients}\label{sec_summing}
As mentioned in the introduction, here we show \eqref{c_pm_sums}, namely that $c^{\pm} = g\sum_{k=0}^{M-1} c^{\pm}_{mk}$, where $c_{mk}^\pm$'s are defined in \eqref{def_of_cmk} and $c^\pm$ are defined in \eqref{def_cs}, that is
\begin{equation*}
\begin{aligned}
c^{+}(\alpha) & = \frac{g a^{2}(2i\alpha + 1)}{(a+i)^{2}}\coth(\pi B_{+}(\alpha) /M), && \quad \alpha>0, \ \alpha\neq1/2a,\\
c^{-}(\alpha) &  = \frac{ga^{2}(-2i\alpha + 1)}{(a+i)^{2}}\coth(\pi B_{-}(\alpha) /M), && \quad \alpha>0.
\end{aligned}
\end{equation*}

To this end, we first note that
\eqnb\label{obs_sum_Bmk}
\sum_{k=0}^{M-1} \mathcal{B}^{\pm}_{mk}(\alpha) = \sinh (\pi B_{\pm})\coth(\pi B_{\pm} /M),
\eqne
for $m=0,\ldots,M-1$ (recall \eqref{def_Bmk} for the definition of $\mathcal{B}_{mk}^\pm$).

Indeed, setting $B\coloneqq B_{\pm}$ we see that
\begin{align*}
 \sum_{k=0}^{M-1} \mathcal{B}^{\pm}_{mk} & = \sum_{k=0}^{M-1}\ee^{B(\theta_{k}-\theta_{m})}
\left\{\begin{aligned}
&\ee^{-\pi B}, && \text{if} \ \theta_{k}>\theta_{m},\\
&\cosh(\pi B), && \text{if} \ \theta_{k}=\theta_{m},\\
&\ee^{\pi B}, && \text{if} \ \theta_{k}<\theta_{m},
\end{aligned}\right. \\
&=\sum_{k=0}^{m-1} \ee^{2\pi B (k-m)/M}\ee^{\pi B} + \cosh(\pi B) + \sum_{k=m+1}^{M-1} \ee^{2\pi B(k-m)/M}\ee^{-\pi B} \\
&=\ee^{-2\pi B m/M}\sum_{k=0}^{m-1} \ee^{2\pi B k/M}\ee^{\pi B} + \cosh(\pi B) + \sum_{k=1}^{M-m-1} \ee^{2\pi B k/M}\ee^{-\pi B} \\
&=\ee^{-2\pi B m/M}\frac{1-\ee^{2\pi B m/M}}{1-\ee^{2\pi B /M}}\ee^{\pi B} + \cosh(\pi B ) + \left(\frac{1-\ee^{2\pi B (M-m)/M}}{1-\ee^{2\pi B/M}}-1\right)\ee^{-\pi B} \\
&=\frac{\ee^{-2\pi B m/M}-1}{1-\ee^{2\pi B/M}}\ee^{\pi B} + \cosh(\pi B ) + \left(\frac{1-\ee^{2\pi B (M-m)/M}}{1-\ee^{2\pi B/M}}-1\right)\ee^{-\pi B} \\
&=-\frac{\ee^{\pi B}}{1-\ee^{2\pi B /M}} + \frac{\ee^{\pi B}+\ee^{-\pi B}}{2} + \frac{\ee^{-\pi B}}{1-\ee^{2\pi B /M}}-\ee^{-\pi B} \\
&=-\frac{\ee^{\pi B}-\ee^{-\pi B}}{1-\ee^{2\pi B /M}} + \frac{\ee^{\pi B}-\ee^{-\pi B}}{2} = \sinh (\pi B) \frac{1+\ee^{2\pi B /M}}{\ee^{2\pi B /M}-1} \\
&=\sinh (\pi B) \coth(\pi B/M),
\end{align*}
which is \eqref{obs_sum_Bmk}, as required.

Applying \eqref{obs_sum_Bmk} in the definition \eqref{def_of_cmk} of the $c_{mk}^\pm$'s immediately gives \eqref{c_pm_sums}.

\section{Finding eigenvalues of the ODE matrix}\label{sec_imag_part}

Here we prove the last claim of Theorem~\ref{thm_main}. Namely we let $M\geq 3$, and we  show that $c_0$ does not lie on the line passing through $c^+$, $c^-$ for all sufficiently large $a>0$ and $\alpha \in (0,\infty ) \setminus \{ 1/2a \}$ sufficiently close to $1/2a$.

To this end we will prove the same claim but with coefficients $c_0$, $c^+$, $c^-$ replaced by their multiples by $(a+i)^2(a^2g)^{-1}$, namely we set
\[
\begin{split}
b_0 &\coloneqq c_0 \frac{(a+i)^2 }{a^2 g} = (1-ia^{-1}) \coth (-2a\pi i/M(a+i) ),\\
b^+ (\alpha )&\coloneqq c^+ \frac{(a+i)^2 }{a^2 g} = (1+2\alpha i )  \coth (\pi B_+ /M) ,\\
b^- (\alpha )&\coloneqq c^- \frac{(a+i)^2 }{a^2 g} = (1-2\alpha i )  \coth (\pi B_- /M)
\end{split}
\]
(recall definitions \eqref{def_cs}--\eqref{def_of_A_Bpm}). Taking the limit $\alpha \to 1/2a$ we obtain
\[
\begin{split}
b^+ & = (1+ia^{-1} )  \coth (-i\pi  /M)  ,\\
b^- & = (1-ia^{-1} )  \coth \left( \left(-\frac{2a}{a^2+1} - i \frac{a^2-1}{a^2+1} \right) \frac{\pi}M \right)
\end{split}
\]
It is easy to see that $b_0$, $b^+$, $b^-$ (after taking the limit $\alpha \to 1/2a^{-}$) are smooth functions of $1/a$, and applying a Taylor expansion at $1/a=0^+$, one obtains as $a\to \infty$
\[
\begin{split}
b_0 &=  i \cot \frac{2\pi }M - a^{-1} \left( \frac{2\pi}{M \sin^2 \frac{2\pi}M }  -\cot\frac{2\pi}M \right) + o(a^{-1}),\\
b^+ &=  i   \cot \frac{\pi}{M} - a^{-1} \cot\frac{\pi}{M},\\
b^- &=  i \cot \frac{\pi}M -  a^{-1} \left(\frac{2\pi}{ M\,\sin^2 \frac{\pi }{M} } -\cot \frac{\pi}M \right) + o(a^{-1}) ,
\end{split}
\]
as it is depicted on Figure~\ref{fig_bs}.
\begin{center}
\includegraphics[width=0.5\textwidth]{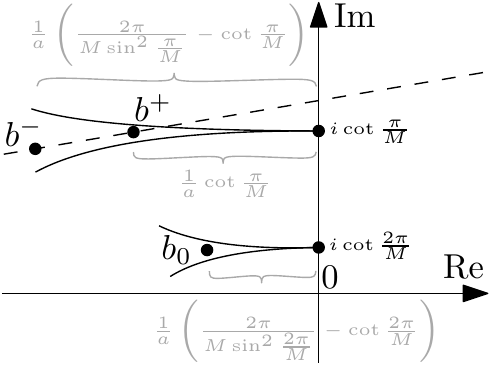}
  \captionof{figure}{A sketch of the locations of $b_0$, $b^-$, $b^+$, for sufficiently large $a>0$. Note that, in the case $M=3$ we have $\cot (2\pi/3)<0$ and so in such case the location of $b_0$ is different than suggested by the sketch.}\label{fig_bs}
  \end{center}
The claim is now clear from the figure. Indeed, the line passing through $b^+$, $b^-$ becomes horizontal in the limit $a\to \infty$, and so in particular it stays away from $b_0$. Thus $b_0$ does not lie on the line passing through $b^+$, $b^-$ for sufficiently large $a>0$. In such case we pick $\alpha$ sufficiently close to $1/2a$ so that the claim remains true by continuity.

\section*{Acknowledgements}
T.C. was partially supported by the National Science Centre grant SONATA BIS 7 number UMO-2017/26/E/ST1/00989. WSO was supported in part by the Simons Foundation.

\section{Appendix}\label{appendix}
Here we show some more properties of the notion of the finite part (recall \eqref{fp}), and we prove \eqref{eq-fp-pv} in Proposition~\ref{prop_fp_vs_pv} below.

\begin{lemma}[Finite part of a regular or p.v. integral]
If $f(t) = g(t)(t-x)^{2}$ then
\begin{align*}
\mathrm{f.p.}\int_{a}^{b} \frac{f(t)}{(t-x)^{2}}\,\d t = \int_{a}^{b} g(t)\,\d t,
\end{align*}
and if $f(t) = g(t)(t-x)$ then
\begin{align*}
\mathrm{f.p.}\int_{a}^{b} \frac{f(t)}{(t-x)^{2}}\,\d t = \mathrm{p.v.}\int_{a}^{b}\frac{g(t)}{t-x}\,\d t
\end{align*}
for every $x\in(a,b)$ and every $g\in C^{\infty}((a,b),\R)$.
\end{lemma}
\begin{proof}
For the first claim we recall the definition \eqref{fp} to get
\begin{align*}
\mathrm{f.p.}\int_{a}^{b} \frac{f(t)}{(t-x)^{2}}\,\d t & = \lim_{\ve\to 0^{+}} \left(\int_{a}^{x-\ve}\frac{g(t)(t-x)^{2}}{(t-x)^{2}}\,\d t +
\int_{x+\ve}^{b}\frac{g(t)(t-x)^{2}}{(t-x)^{2}}\,\d t-\frac{g(x+\ve)\ve^{2}+g(x-\ve)(-\ve)^{2}}{\ve}\right) \\
&=\lim_{\ve\to 0^{+}} \left(\int_{a}^{x-\ve} g(t)\,\d t +
\int_{x+\ve}^{b}g(t) \,\d t-(g(x+\ve)\ve-g(x-\ve)\ve)\right) = \int_{a}^{b}g(t)\,\d t ,
\end{align*}
and for the second claim we have
\begin{align*}
&\mathrm{f.p.}\int_{a}^{b} \frac{f(t)}{(t-x)^{2}}\,\d t  = \lim_{\ve\to 0^{+}} \left(\int_{a}^{x-\ve}\frac{g(t)(t-x)}{(t-x)^{2}}\,\d t +
\int_{x+\ve}^{b}\frac{g(t)(t-x)}{(t-x)^{2}}\,\d t-\frac{g(x+\ve)\ve+g(x-\ve)(-\ve)}{\ve}\right) \\
&\qquad = \lim_{\ve\to 0^{+}} \left(\int_{a}^{x-\ve}\frac{g(t)}{t-x}\,\d t + \int_{x+\ve}^{b}\frac{g(t)}{t-x}\,\d t-(g(x+\ve) - g(x-\ve))\right)
=\mathrm{p.v.}\int_{a}^{b}\frac{g(t)}{t-x}\,\d t.
\end{align*}
as desired.
\end{proof}
We now consider a bit more general denominator, and we characterize the finite part from Definition \ref{def_fp_curve} as the derivative of the principal value integral.
\begin{lemma}[FP vs. PV]\label{lem-diff}
For any $x\in(a,b)$, the following equality holds
\begin{align}
\frac{\d}{\d x}\left(\mathrm{p.v.}\int_{a}^{b}\frac{f(t)}{\gamma(t)-\gamma(x)}\,\d t\right) = \gamma'(x)\,\mathrm{f.p.}\int_{a}^{b} \frac{f(t)}{(\gamma(t)-\gamma(x))^{2}}\,\d t,
\end{align}
where $\gamma:[a,b]\to \C$ is a smooth curve such that $\gamma'(t)\neq 0$ for $t\in[a,b]$.
\end{lemma}
\begin{proof}
Let us observe that
\begin{align*}
&\frac{\d}{\d x}\left(\int_{a}^{x-\ve}\frac{f(t)}{\gamma(t)-\gamma(x)}\,\d t + \int_{x+\ve}^{b}\frac{f(t)}{\gamma(t)-\gamma(x)}\,\d t\right) =
\int_{a}^{x-\ve}\frac{\gamma'(x)f(t)\,\d t}{(\gamma(t)-\gamma(x))^{2}} + \int_{x+\ve}^{b}\frac{\gamma'(x)f(t)\,\d t}{(\gamma(t)-\gamma(x))^{2}} \\
&\quad +\frac{f(x-\ve)}{\gamma(x-\ve)-\gamma(x)} - \frac{f(x+\ve)}{\gamma(x+\ve)-\gamma(x)}.
\end{align*}
Then we have the following asymptotics as $\ve\to 0$
\begin{align*}
&-\frac{\ve f(x+\ve)\gamma'(x)}{(\gamma(x+\ve)-\gamma(x))^{2}} + \frac{f(x+\ve)}{\gamma(x+\ve)-\gamma(x)}
 = -\frac{f(x+\ve)(\ve\gamma'(x) - \gamma(x+\ve) + \gamma(x))}{(\gamma(x+\ve)-\gamma(x))^{2}} \\
&\qquad  = -\frac{f(x+\ve)(\frac{1}{2}\gamma''(x)\ve^{2} + o(\ve^{2}))}{(\gamma(x+\ve)-\gamma(x))^{2}}
=-\frac{1}{2}\frac{f(x+\ve)\gamma''(x)\ve^{2}}{(\gamma(x+\ve)-\gamma(x))^{2}} + o(1)
\end{align*}
and furthermore
\begin{align*}
& -\frac{\ve f(x-\ve)\gamma'(x)}{(\gamma(x-\ve)-\gamma(x))^{2}} - \frac{f(x-\ve)}{\gamma(x-\ve)-\gamma(x)} =
\frac{f(x-\ve)(-\ve\gamma'(x) - \gamma(x-\ve) + \gamma(x))}{(\gamma(x-\ve)-\gamma(x))^{2}} \\
&\qquad  = \frac{f(x-\ve)(\frac{1}{2}\gamma''(x)\ve^{2} + o(\ve^{2}))}{(\gamma(x-\ve)-\gamma(x))^{2}}
=\frac{1}{2}\frac{f(x-\ve)\gamma''(x)\ve^{2}}{(\gamma(x-\ve)-\gamma(x))^{2}} + o(1).
\end{align*}
Since the limits in the definition of $\mathrm{p.v.}$ and $\mathrm{f.p.}$ are uniform with respect to $x$ in compact subsets of $(a,b)$, it follows that
\begin{align*}
&\frac{\d}{\d x}\left(\mathrm{p.v.}\int_{a}^{b}\frac{f(t)}{\gamma(t)-\gamma(x)}\,\d t\right) - \gamma'(x)\,\mathrm{f.p.}\int_{a}^{b} \frac{f(t)}{(\gamma(t)-\gamma(x))^{2}}\,\d t\\
&\qquad = \lim_{\ve\to 0}\left(-\frac{1}{2}\frac{f(x+\ve)\gamma''(x)\ve^{2}}{(\gamma(x+\ve)-\gamma(x))^{2}} + \frac{1}{2}\frac{f(x-\ve)\gamma''(x)\ve^{2}}{(\gamma(x-\ve)-\gamma(x))^{2}} \right) = 0,
\end{align*}
and the proof of lemma is completed. 
\end{proof}
In the following proposition we prove the integration by parts formula \eqref{eq-fp-pv}.
\begin{proposition}[FP vs. PV via integration by parts]\label{prop_fp_vs_pv}
Let us assume that $\gamma:[a,b]\to\C$ is a smooth curve satisfying $\gamma'(t)\neq 0$ for $t\in[a,b]$. Then, for any $x\in(a,b)$, we have
\begin{align}\label{eq-fp-pv_repeat}
\mathrm{f.p.}\int_{a}^{b} \frac{f(t)\,\d t}{(\gamma(t)\!-\!\gamma(x))^{2}} = \mathrm{p.v.}\int_{a}^{b}\frac{\left(f(t)/\gamma'(t)\right)'\d t}{\gamma(t)-\gamma(x)} - \frac{f(b)}{\gamma'(b)(\gamma(b)\!-\!\gamma(x))} + \frac{f(a)}{\gamma'(a)(\gamma(a)\!-\!\gamma(x))}.
\end{align}
\end{proposition}
\begin{proof}
Let us define
\begin{align*}
I(\ve)\coloneqq \frac{\ve}{(\gamma(x+\ve)-\gamma(x))^{2}}\quad\text{and}\quad J(\ve)\coloneqq  \frac{1}{\gamma'(x+\ve)(\gamma(x+\ve)-\gamma(x))},\quad \ve\in\R\setminus\{0\}.
\end{align*}
Then we have
\begin{align*}
&\int_{a}^{x-\ve}\frac{f(t)\,\d t}{(\gamma(t)-\gamma(x))^{2}} = -\int_{a}^{x-\ve}\frac{f(t)}{\gamma'(t)}\left(\frac{1}{\gamma(t)-\gamma(x)}\right)'\,\d t \\
&\qquad = \int_{a}^{x-\ve} \left(\frac{f(t)}{\gamma'(t)}\right)'\frac{\d t}{\gamma(t)-\gamma(x)}
-f(x-\ve)J(-\ve) + \frac{f(a)}{\gamma'(a)(\gamma(a)-\gamma(x))}
\end{align*}
and similarly
\begin{align*}
\int_{x+\ve}^{b}\frac{f(t)\,\d t}{(\gamma(t)-\gamma(x))^{2}} 
= \int_{x+\ve}^{b} \left(\frac{f(t)}{\gamma'(t)}\right)'\frac{dt}{\gamma(t)-\gamma(x)}
+f(x+\ve)J(\ve) - \frac{f(b)}{\gamma'(b)(\gamma(b)-\gamma(x))}.
\end{align*}
Let us write
\begin{align*}
R(\ve)&\coloneqq  -f(x+\ve)I(\ve) + f(x-\ve)I(-\ve) +f(x+\ve)J(\ve) - f(x-\ve)J(-\ve),\quad \ve\in \R\setminus\{0\}.
\end{align*}
By the formula \eqref{eq-pv-2}, to prove \eqref{eq-fp-pv_repeat}, it is enough to show that $R(\ve) = o(1)$ as $\ve\to 0$. To this end we observe that
\begin{equation}
\begin{aligned}\label{eq-p-1}
R(\ve) & = f(x+\ve)(J(\ve)-I(\ve)) + f(x-\ve)(I(-\ve) - J(-\ve)) \\
& = f(x)(J(\ve)-I(\ve)) + f(x)(I(-\ve) - J(-\ve)) \\
&\quad - O(\ve)(J(\ve)-I(\ve)) +  O(\ve)(I(-\ve) - J(-\ve))\quad \text{ as }\quad \ve\to 0.
\end{aligned}
\end{equation}
Then we have
\begin{equation}
\begin{aligned}\label{eq-p-2}
J(\ve)-I(\ve) & = \frac{\gamma(x+\ve)-\gamma(x) - \ve\gamma'(x+\ve)}{\gamma'(x+\ve)(\gamma(x+\ve)-\gamma(x))^{2}}
= \frac{\gamma'(x)\ve +o(\ve) - \ve(\gamma'(x)+\gamma''(x)\ve + o(\ve))}{\gamma'(x+\ve)(\gamma(x+\ve)-\gamma(x))^{2}} \\
& = \frac{o(\ve) - \gamma''(x)\ve^{2} + \ve o(\ve)}{\gamma'(x+\ve)(\gamma(x+\ve)-\gamma(x))^{2}},\quad \ve\to 0,
\end{aligned}
\end{equation}
which after changing of variables $\ve\mapsto -\ve$ gives
\begin{equation}
\begin{aligned}\label{eq-p-3}
I(-\ve) - J(-\ve) = \frac{o(\ve) - \gamma''(x)\ve^{2} - \ve o(\ve)}{\gamma'(x-\ve)(\gamma(x-\ve)-\gamma(x))^{2}}, \quad \ve\to 0.
\end{aligned}
\end{equation}
Combining \eqref{eq-p-1}, \eqref{eq-p-2} and \eqref{eq-p-3} provides
\begin{align}\label{eq-p-4}
R(\ve) = f(x)(J(\ve)-I(\ve)) + f(x)(I(-\ve) - J(-\ve)) + o(1),\quad \ve\to 0.
\end{align}
Let us observe that
\begin{equation}
\begin{aligned}\label{eq-p-9}
& (J(\ve)-I(\ve) + I(-\ve) - J(-\ve))(\gamma(x-\ve)-\gamma(x))^{2}(\gamma(x+\ve)-\gamma(x))^{2}\gamma'(x+\ve)\gamma'(x-\ve) \\
& \quad= -(\gamma(x-\ve)-\gamma(x))(\gamma(x+\ve)-\gamma(x))^{2}\gamma'(x+\ve) \\
& \qquad +(\gamma(x-\ve)-\gamma(x))^{2}(\gamma(x+\ve)-\gamma(x))\gamma'(x-\ve) \\
& \qquad -\ve(\gamma(x-\ve)-\gamma(x))^{2}\gamma'(x+\ve)\gamma'(x-\ve) \\
& \qquad -\ve(\gamma(x+\ve)-\gamma(x))^{2}\gamma'(x+\ve)\gamma'(x-\ve) \\
& = K_{1}(\ve) - K_{1}(-\ve) + K_{2}(\ve)-K_{2}(-\ve)
\end{aligned}
\end{equation}
for $\ve\in\R\setminus\{0\}$, where we define
\begin{align*}
K_{1}(\ve) & := -(\gamma(x-\ve)-\gamma(x))(\gamma(x+\ve)-\gamma(x))^{2}\gamma'(x+\ve) \\
K_{2}(\ve) & := -\ve(\gamma(x-\ve)-\gamma(x))^{2}\gamma'(x+\ve)\gamma'(x-\ve).
\end{align*}
We show that $K_{1}(\ve) - K_{1}(-\ve) + K_{2}(\ve)-K_{2}(-\ve)=o(\ve^{4})$ as $\ve\to 0$. To this end, we observe that
\begin{align*}
K_{1}(\ve) &= - (-\ve\gamma'(x) + \gamma''(x)\ve^{2}/2 + o(\ve^{2}))(\ve\gamma'(x)+\gamma''(x)\ve^{2}/2 + o(\ve^{2}))^{2}(\gamma'(x) + \gamma''(x)\ve+o(\ve))\\
& = - (-\ve\gamma'(x) + \gamma''(x)\ve^{2}/2)(\ve\gamma'(x) + \gamma''(x)\ve^{2}/2 + o(\ve^{2}))^{2}(\gamma'(x) + \gamma''(x)\ve) + o(\ve^{4}) \\
& = - (-\ve\gamma'(x) + \gamma''(x)\ve^{2}/2)(\ve\gamma'(x) + \gamma''(x)\ve^{2}/2)^{2}(\gamma'(x) + \gamma''(x)\ve) + o(\ve^{4}) \\
& = - (-\ve^{2}\gamma'(x)^{2} + \gamma''(x)^{2}\ve^{4}/4)(\ve\gamma'(x) + \gamma''(x)\ve^{2}/2)(\gamma'(x) + \gamma''(x)\ve) + o(\ve^{4}) \\
& = \ve^{2}\gamma'(x)^{2}(\ve\gamma'(x) + \gamma''(x)\ve^{2}/2)(\gamma'(x) + \gamma''(x)\ve) + o(\ve^{4})  \\
& = \ve^{2}\gamma'(x)^{3}(\ve\gamma'(x) + \gamma''(x)\ve^{2}/2) + \ve^{4}\gamma'(x)^{2}(\gamma'(x) + \gamma''(x)\ve/2) \gamma''(x) + o(\ve^{4}),
\end{align*}
as $\ve\to 0$, and consequently
\begin{equation}
\begin{aligned}\label{eq-p-5}
& K_{1}(\ve) - K_{1}(-\ve)= \ve^{2}\gamma'(x)^{3}(\ve\gamma'(x) + \gamma''(x)\ve^{2}/2) - \ve^{2}\gamma'(x)^{3}(-\ve\gamma'(x) + \gamma''(x)\ve^{2}/2)\\
&\quad + \ve^{4}\gamma'(x)^{2}(\gamma'(x)+\gamma''(x)\ve/2)\gamma''(x)-\ve^{4}\gamma'(x)^{2}(\gamma'(x)-\gamma''(x)\ve/2) \gamma''(x) + o(\ve^{4}) \\
& = \ve^{3}\gamma'(x)^{4} + \ve^{4}\gamma'(x)^{3}\gamma''(x)/2 + \ve^{3}\gamma'(x)^{4} - \ve^{4}\gamma'(x)^{3}\gamma''(x)/2 + o(\ve^{4}) \\
& = 2\ve^{3}\gamma'(x)^{4} + o(\ve^{4})\quad \text{ as }\quad \ve\to 0.
\end{aligned}
\end{equation}
Furthermore we have
\begin{align*}
K_{2}(\ve) & = -\ve(-\gamma'(x)\ve + \gamma''(x)\ve^{2}/2 + o(\ve^{2}))^{2}(\gamma'(x) + \gamma''(x)\ve + o(\ve))(\gamma'(x) - \gamma''(x)\ve+o(\ve))\\
& = -\ve(-\gamma'(x)\ve + \gamma''(x)\ve^{2}/2 + o(\ve^{2}))^{2}(\gamma'(x) + \gamma''(x)\ve)(\gamma'(x) - \gamma''(x)\ve) + o(\ve^{4}) \\
& = -\ve(-\gamma'(x)\ve + \gamma''(x)\ve^{2}/2 + o(\ve^{2}))^{2}\gamma'(x)^{2} + o(\ve^{4}) \\
& = -\ve(-\gamma'(x)\ve + \gamma''(x)\ve^{2}/2 )^{2}\gamma'(x)^{2} + o(\ve^{4}) \\
& = -\ve^{3}\gamma'(x)^{4} +\ve^{4}\gamma'(x)^{3}\gamma''(x) + o(\ve^{4})\quad \text{ as }\quad \ve\to 0,
\end{align*}
which implies that
\begin{equation}
\begin{aligned}\label{eq-p-6}
K_{2}(\ve)-K_{2}(-\ve)& =-\ve^{3}\gamma'(x)^{4} +\ve^{4}\gamma'(x)^{3}\gamma''(x)-\ve^{3}\gamma'(x)^{4}-\ve^{4}\gamma'(x)^{3}\gamma''(x)+o(\ve^{4}) \\
& = -2\ve^{3}\gamma'(x)^{4}+ o(\ve^{4})\quad \text{ as }\quad \ve\to 0.
\end{aligned}
\end{equation}
Combining \eqref{eq-p-4}, \eqref{eq-p-9}, \eqref{eq-p-5} and \eqref{eq-p-6} we obtain
\begin{align*}
R(\ve) & 
 = \frac{K_{1}(\ve) - K_{1}(-\ve) + K_{2}(\ve)-K_{2}(-\ve)}{(\gamma(x-\ve)-\gamma(x))^{2}(\gamma(x+\ve)-\gamma(x))^{2}\gamma'(x+\ve)\gamma'(x-\ve)} + o(1) \\
& = \frac{o(\ve^{4})}{(\gamma(x-\ve)-\gamma(x))^{2}(\gamma(x+\ve)-\gamma(x))^{2}\gamma'(x+\ve)\gamma'(x-\ve)} + o(1) = o(1)
\end{align*}
as $\ve\to0$ and the proof of the proposition is completed.
\end{proof}

\end{document}